\documentclass[12pt]{amsart}
\usepackage{amsmath}
\usepackage{amssymb,latexsym,amsthm,amsfonts,amscd}
\usepackage{graphicx}
\usepackage[pdftex]{hyperref}
\newtheorem{theorem}{Theorem}[section]
\newtheorem{lemma}[theorem]{Lemma}
\newtheorem{proposition}[theorem]{Proposition}
\newtheorem{corollary}[theorem]{Corollary}

\newtheorem{remarks}[theorem]{Remarks}

\newtheorem{definition}[theorem]{Definition}

\numberwithin{equation}{section}

\begin{document}

\baselineskip=15.5pt

\title[Skeleton Ideals, Standard Monomials and 
Spherical Parking Functions]{Skeleton Ideals of Certain Graphs, Standard Monomials and 
Spherical Parking Functions}

\author[C. Kumar]{Chanchal Kumar}

\address{IISER Mohali,
Knowledge City, Sector 81, SAS Nagar, Punjab -140 306, India.}

\email{chanchal@iisermohali.ac.in}

\author[G. Lather]{Gargi Lather}

\address{IISER Mohali,
Knowledge City, Sector 81, SAS Nagar, Punjab -140 306, India.}

\email{gargilather@iisermohali.ac.in}

\author[Sonica]{Sonica}

\address{MCM DAV College for Women, Sector- 36 A, Chandigarh
- 160 036, India.}

\email{sonica.anand@gmail.com}

\date{}

\begin{abstract}
Let $G$ be an (oriented) graph on the vertex set 
$V = \{ 0, 1,\ldots,n\}$ with root $0$. Postnikov 
and Shapiro associated a monomial ideal $\mathcal{M}_G$
in the polynomial ring $ R = {\mathbb{K}}[x_1,\ldots,x_n]$
over a field $\mathbb{K}$. The standard monomials
of the Artinian quotient $\frac{R}{\mathcal{M}_G}$
correspond bijectively to $G$-parking functions. A subideal
$\mathcal{M}_G^{(k)}$ of $\mathcal{M}_G$ generated by 
subsets of $\widetilde{V}=V\setminus \{0\}$ of 
size at most $k+1$ is called a $k$-{\em skeleton ideal of the graph}
$G$.
Many interesting homological and combinatorial properties
of $1$-skeleton ideal $\mathcal{M}_G^{(1)}$ are obtained by
 Dochtermann for certain classes of simple graph $G$.  A 
 finite sequence $\mathcal{P}=(p_1,\ldots,p_n) \in 
 \mathbb{N}^n$ is called a {\em spherical $G$-parking function}
   if the monomial 
  $\mathbf{x}^{\mathcal{P}} = \prod_{i=1}^{n} x_i^{p_i}
  \in \mathcal{M}_G \setminus \mathcal{M}_G^{(n-2)}$. Let
  ${\rm sPF}(G)$ be the set of all spherical $G$-parking functions.
 On counting the number of 
spherical parking functions of a 
 complete graph $K_{n+1}$ on $V$, in two different
 ways, Dochtermann obtained a new identity for $(n-1)^{n-1}$.
  In this paper, a combinatorial 
  description for all multigraded Betti numbers of the
  $k$-skeleton ideal $\mathcal{M}_{K_{n+1}}^{(k)}$ of
  the complete graph $K_{n+1}$ on $V$ are given. Also, using DFS burning 
  algorithms of Perkinson-Yang-Yu (for simple graph) and 
  Gaydarov-Hopkins (for multigraph), we give a combinatorial
  interpretation of spherical $G$-parking functions for the graph 
$G = K_{n+1}- \{e\}$ obtained from the complete 
graph $K_{n+1}$ on deleting an edge $e$. In particular,
we showed that $|{\rm sPF}(K_{n+1}- \{e_0\} )|= (n-1)^{n-1}$
for an edge $e_0$ through the root $0$, but
 $|{\rm sPF}(K_{n+1} - \{e_1\})| = (n-1)^{n-3}(n-2)^2$ 
for an edge $e_1$  not through the root.
\vspace{0.2cm}

\end{abstract}
\maketitle
\section{Introduction}
Let $G$ be an oriented graph on the vertex set 
$V = \{0,1,\ldots,n\}$ with a root ${0}$. 
The graph $G$ is completely determined by an 
$ (n+1) \times (n+1)$ matrix $A(G) = [a_{ij}]_{0 \le i,j \le n}$, 
called its {\em adjacency matrix}, where $a_{ij}$ is the
number of oriented edges from $i$ to $j$. A non-oriented
graph $\tilde{G}$ on $V$ (rooted at $0$)
is identified with the unique (rooted) oriented graph $G$
on $V$ with symmetric adjacency matrix $A(\tilde{G})= A(G)$.
Let $R = {\mathbb{K}}[x_1,\ldots,x_n]$ be the standard polynomial
ring in $n$ variables over a field $\mathbb{K}$. The $G$-{\em parking function
ideal} $\mathcal{M}_G$ of $G$ is a monomial ideal in $R$ given by the 
generating set  
\[
\mathcal{M}_G = \left\langle m_A =  \prod_{i \in A} x_i^{d_{A}(i)} : 
\emptyset \neq A  \subseteq [n]=\{1,\ldots,n\} \right\rangle ,
\]
where $d_A(i) = \sum_{j \in V \setminus A} a_{ij} $ is the number of edges
from $i$ to a vertex outside the set $A$ in $G$.
A standard monomial basis $\{ \mathbf{x}^{\mathbf{b}} = 
\prod_{i=1}^n x_i^{b_i} \}$ of the
the Artinian quotient
$\frac{R}{{\mathcal{M}}_G}$ is determined by the set  
$ {\rm PF}(G)=\{ \mathbf{b}=(b_1,\ldots,b_n) \in \mathbb{N}^n : 
\mathbf{x}^{\mathbf{b}} \notin
\mathcal{M}_G\}$ of $G$-parking functions. Further, 
$\dim_{\mathbb{K}}\left(\frac{R}{\mathcal{M}_G} \right)$ is the number 
of (oriented) spanning trees of $G$, given by the determinant
$\det(L_G)$ of the truncated Laplace matrix $L_G$ of $G$.
Let ${\rm SPT}(G)$ be the set of (oriented) spanning trees of $G$.
If $G$ is non-oriented, then the edges of a spanning tree of $G$
is given orientation so that all paths 
in the spanning tree
are directed away from the 
root. In this paper, $G$ is assumed to be a non-oriented graph.
For $G = K_{n+1}$, the complete graph on $V$, the 
$K_{n+1}$-parking functions are the {\em ordinary parking functions} of 
length $n$.
 As $|{\rm PF}(G)|=
|{\rm SPT}(G)|$, one would like to construct an explicit
bijection $\phi : {\rm PF}(G) \longrightarrow {\rm SPT}(G)$.
Using a depth-first search (DFS) version of burning algorithm, 
 an algorithmic bijection 
$\phi : {\rm PF}(G) \longrightarrow {\rm SPT}(G)$ for simple
 graphs $G$ preserving 
{reverse (degree) sum} of $G$-parking function 
$\mathcal{P}$ and the number of
$\kappa$-inversions of the spanning tree $\phi(\mathcal{P})$ is constructed by
Perkinson, Yang and Yu \cite{PYY}. A similar bijection
for multigraphs $G$ is constructed by Gaydarov and Hopkins \cite{GH}.
 
Postnikov and Shapiro \cite{PoSh} introduced the ideal $\mathcal{M}_G$
and derived many of its combinatorial and homological properties. In 
particular, they showed that the minimal resolution of 
$\mathcal{M}_G$
is the cellular resolution supported on the first barycentric subdivision
${\mathbf{Bd}}(\Delta_{n-1})$ of an $n-1$-simplex $\Delta_{n-1}$, provided
$G$ is saturated (i.e., $a_{ij}>0$ for $i \neq j$). The minimal resolution
of $\mathcal{M}_G$ for any graph $G$ is described in \cite{DS,MSW,MoSh}.

Dochtermann \cite{D1,D2} proposed to investigate 
subideals of the $G$-parking function
ideal $\mathcal{M}_G$ described by $k$-dimensional `skeleta'. 
For an integer $k$ ($0 \le k \le n-1$), the
 $k$-{\em skeleton ideal} $\mathcal{M}_G^{(k)}$
of the graph $G$ is defined as the subideal
\[
 \mathcal{M}_G^{(k)} = \left\langle m_A =  \prod_{i \in A} x_i^{d_{A}(i)}
 : \emptyset \ne A \subseteq [n]; |A| \le k+1 \right\rangle
\]
of the monomial ideal $\mathcal{M}_G$. For $k=0$, 
the ideal $\mathcal{M}_G^{(0)}$ is
generated by powers of variables
$x_1,\ldots,x_n$. Hence, its minimal free resolution and the number of
standard monomials can be easily determined. For $k=1$ and $G=K_{n+1}$,
the minimal resolution of one-skeleton ideal 
$\mathcal{M}_{K_{n+1}}^{(1)}$ is a cocellular resolution supported on the 
labelled polyhedral complex induced by any generic arrangement  of two 
tropical hyperplanes in $\mathbb{R}^n$  and 
$i^{th}$ Betti number $\beta_i\left(\frac{R}{\mathcal{M}_{K_{n+1}}^{(1)}}\right)=
\sum_{j=1}^n ~ j {j-1 \choose i-1}$ for $1 \le i \le n-1$
(see \cite{D2}). Also, the number of 
standard monomials of $\frac{R}{\mathcal{M}_{K_{n+1}}^{(1)}}$ is given by
\[
 \dim_{\mathbb{K}}\left(\frac{R}{\mathcal{M}_{K_{n+1}}^{(1)}}\right) 
= (2n-1)(n-1)^{n-1} = \det(Q_{K_{n+1}}),
\]
where $Q_{K_{n+1}}$ is the truncated signless Laplace matrix of $K_{n+1}$.

A notion of spherical $G$-parking functions is introduced in \cite{D1} for 
the complete graph $G = K_{n+1}$. Let ${\rm PF}(G) = \{ \mathcal{P} \in \mathbb{N}^n :
\mathbf{x}^{\mathcal{P}} \notin \mathcal{M}_G\}$ be the set of $G$-parking 
functions. Consider the set ${\rm sPF}(G) = \{ \mathcal{P} \in \mathbb{N}^n : 
\mathbf{x}^{\mathcal{P}} \in \mathcal{M}_G \setminus \mathcal{M}_G^{(n-2)}\}$. 
The standard monomials of $\frac{R}{\mathcal{M}_G^{(n-2)}}$ are 
$\mathbf{x}^{\mathcal{P}}$ for $\mathcal{P} \in {\rm PF}(G)$ or
 $\mathcal{P} \in {\rm sPF}(G)$. Thus, 
$\dim_{\mathbb{K}}\left(\frac{R}{\mathcal{M}_G^{(n-2)}}\right) = 
\dim_{\mathbb{K}}\left(\frac{R}{\mathcal{M}}\right) + 
\dim_{\mathbb{K}}\left(\frac{\mathcal{M}_G}{\mathcal{M}_G^{(n-2)}}\right) = 
|{\rm PF}(G)|+|{\rm sPF}(G)|$. 
A finite sequence $\mathcal{P}=(p_1,\ldots,p_n)
\in {\rm sPF}(G)$ is called a {\em spherical $G$-parking function}.
A $G$-parking or a spherical $G$-parking function
 $\mathcal{P}=(p_1,\ldots,p_n) \in \mathbb{N}^n$ can be equivalently 
thought of as a function $\mathcal{P}: \widetilde{V} \longrightarrow \mathbb{N}$ 
on $\widetilde{V}=V\setminus \{0\}$ with
$\mathcal{P}(i) = p_i \quad (1 \le i \le n)$. 
The {\em sum} (or {\em degree}) of $\mathcal{P}$ is given by
${\rm sum}(\mathcal{P}) = \sum_{i \in \widetilde{V}} \mathcal{P}(i)$.
We recall that a $K_{n+1}$-parking function 
$\mathcal{P}=(p_1, \ldots,p_n) \in \mathbb{N}^n$ is an 
{\em ordinary parking function} of length $n$, i.e.,
a non-decreasing rearrangement $p_{i_1} \le p_{i_2} \le
\ldots \le p_{i_n}$ of $\mathcal{P}=(p_1,\ldots,p_n)$
satisfies $p_{i_j} <j,~\forall j$. It can be easily checked 
that $\mathcal{P}=(p_1,\ldots,p_n) \in \mathbb{N}^n$ is a spherical
$K_{n+1}$-parking function if a non-decreasing rearrangement
$p_{i_1} \le p_{i_2} \le
\ldots \le p_{i_n}$ of $\mathcal{P}=(p_1,\ldots,p_n)$
satisfies $p_{i_1}=1$ and  $p_{i_j} <j$ for $2 \le j \le n$.
The notion of spherical $K_{n+1}$-parking function has appeared earlier 
in the literature (see \cite{St2}) as {\em prime parking functions} 
of length $n$. Prime parking functions were defined and enumerated
by I. Gessel. The number of spherical $K_{n+1}$-parking functions
is $(n-1)^{n-1}$, which is same as the number of {\em uprooted trees}
on $[n]$. A (labelled) rooted tree $T$ on $[n]$ is called {\em uprooted}
if the root is bigger than all its children. Let $\mathcal{U}_n$ be the
set of uprooted trees on $[n]$. Dochtermann \cite{D1} conjectured existence of
a bijection  $\phi : {\rm sPF}(K_{n+1}) \longrightarrow \mathcal{U}_n$ 
such that ${\rm sum}(\mathcal{P}) = 
{n \choose 2} - \kappa(K_n, \phi(\mathcal{P})) +1$, where
$\kappa(K_{n}, \phi(\mathcal{P}))$ is the $\kappa$-number 
of the uprooted tree $\phi(\mathcal{P})$ 
in the complete graph $K_n= K_{n+1} - \{0\}$ 
 on $\widetilde{V} = [n]$.

The $k$-skeleton ideal $\mathcal{M}_{K_{n+1}}^{(k)}$ of the complete graph
$K_{n+1}$ can be identified with an Alexander dual of some multipermutohedron ideal.
Let $\mathbf{u}=(u_1,u_2,\ldots,u_n) \in \mathbb{N}^n$ such that
$u_1 \le u_2 \le \ldots \le u_n$. Set $\mathbf{m}=(m_1,\ldots,m_s)$
such that the smallest entry in $\mathbf{u}$
is repeated exactly $m_1$ times, second smallest entry in $\mathbf{u}$ is repeated
exactly 
$m_2$ times, and so on. Then $\sum_{j=1}^s m_j = n$ and $m_j \ge 1$ for all $j$.
In this case, we write $\mathbf{u(\mathbf{m})}$ for $\mathbf{u}$.
Let 
$\mathfrak{S}_n$ be the set of permutations of $[n]$.
For a permutation $\sigma$ of $[n]$, let 
$\mathbf{x}^{\sigma \mathbf{u(\mathbf{m})}} = \prod_{i=1}^n x_i^{u_{\sigma(i)}}$.
The monomial ideal $I(\mathbf{u(\mathbf{m})}) = \langle
\mathbf{x}^{\sigma \mathbf{u}(\mathbf{m})} : \sigma \in \mathfrak{S}_n \rangle$
of $R$ is called a {\em multipermutohedron ideal}. 
If $\mathbf{m}=(1,\ldots,1)\in \mathbb{N}^n$,
then $I(\mathbf{u}({\mathbf{m}}))$ is called a {\em permutohedron ideal}.
 
Let  $\mathbf{u(\mathbf{m})}=
(1,2,\ldots,k,k+1,\ldots,k+1) \in \mathbb{N}^n$, where
$\mathbf{m}=(1,\ldots,1,n-k) \in \mathbb{N}^{k+1}$.
For $1 \le k \le n-1$, the Alexander dual 
$I(\mathbf{u(\mathbf{m})})^{[\mathbf{n}]}$
of multipermutohedron ideal $I(\mathbf{u(\mathbf{m})})$ with respect to
$\mathbf{n}=(n,\ldots,n) \in \mathbb{N}^n$ equals the $k$-skeleton ideal
$\mathcal{M}_{K_{n+1}}^{(k)}$. Thus, the number of standard 
monomials of the Artinian quotient 
$\frac{R}{I(\mathbf{u(\mathbf{m})})^{[\mathbf{n}]}}
= \frac{R}{\mathcal{M}_{K_{n+1}}^{(k)}}$ is given by the number of
$\lambda$-parking functions for 
$\lambda = (n,n-1,\ldots,n-k+1,n-k,\ldots,n-k)\in 
\mathbb{N}^n$ (see \cite{PiSt, PoSh}). 
A description of multigraded Betti numbers
of Alexander duals of multipermutohedron ideals and a 
simple proof of the Steck determinant
formula for the number of $\lambda$-parking functions are given in \cite{AC1}.
We obtain a combinatorial expression for the ${(i-1)}^{th}$ Betti number 
$\beta_{i-1}\left(\mathcal{M}_{K_{n+1}}^{(k)}\right)$ 
(Proposition \ref{Prop1}). In particular, for $n \ge 3$,
we show that
$\beta_{i-1}\left(\mathcal{M}_{K_{n+1}}^{(1)}\right) = 
i {n-1 \choose i+1}$ and 
$ \beta_{i-1}\left(\mathcal{M}_{K_{n+1}}^{(n-2)}\right)$ equals
\[
 \sum_{0 < j_1 < \ldots <j_i < j_{i+1}=n} 
\frac{n!}{j_1! (j_2-j_1)! \cdots (n-j_i)!} + 
\sum_{0 < l_1 < \ldots <l_{i-3} < n-1} 
\frac{n!}{l_1!(l_2-l_1)!\cdots(n-l_{i-3})!}
(n-l_{i-3}-1) .
\]

By a simple modification of DFS algorithm of
Peterson, Yang and Yu \cite{PYY}, we construct an algorithmic bijection
$\phi_n : {\rm sPF}(K_{n+1}) \longrightarrow \mathcal{U}_n$ as conjectured by 
Dochtermann \cite{D1}. Further, we compare spherical parking functions of a
 graph $G$ with that of $G - \{e\}$. If $e$ is an edge of $G$, then 
$G - \{e\}$ is the graph obtained on deleting the edge $e$ from $G$.
We show that $|{\rm sPF}(G)| = |{\rm sPF}(G - \{e_0\})|$ (Lemma \ref{Lem2}), 
where $e_0$
is an edge from the root to another vertex. As an application of this
result, we show that the number of spherical parking functions of a 
complete bipartite graph $K_{m+1,n}$ satisfies
\[
 |{\rm sPF}(K_{m+1,n})| = |{\rm sPF}(K_{n+1,m})|.
\]
We obtain a formula for $|{\rm sPF}(K_{m+1,n})|$ (Proposition \ref{Prop5}),
which is symmetric in $m$ and $n$.
For the complete graph $K_{n+1}$ and an edge $e_1$ not through the root, we show 
that $|{\rm sPF}(K_{n+1} - \{e_1\}) | = (n-1)^{n-3} (n-2)^2$ (Theorem \ref{Thm6}).

 Some extensions
of these results for the complete multigraph $K_{n+1}^{a,b}$ are also obtained.

\section{$k$-skeleton ideals of complete graphs}

Let $0 \le k \le n-1$. Consider the $k$-skeleton ideal 
$\mathcal{M}_{K_{n+1}}^{(k)}$ of the complete graph $K_{n+1}$ on the
vertex set $V=\{0,1,\ldots,n\}$. As stated in the Introduction, we have
\[
 \mathcal{M}_{K_{n+1}}^{(k)} = 
\left\langle \left( \prod_{j \in A} x_j \right)^{n-|A|+1} : 
\emptyset \ne A \subseteq [n]; ~ |A| \le k+1 \right\rangle.
\]
For $k=0$, $\mathcal{M}_{K_{n+1}}^{(0)} = \langle x_1^n,\ldots,x_n^n \rangle$ is a 
monomial ideal in $R$ generated by $n^{th}$ power of variables. 
Thus, its minimal free resolution
is given by the Koszul complex associated to the regular 
sequence $x_1^n, \ldots, x_n^n$ in $R$. Also,
$\dim_{\mathbb{K}}\left(\frac{R}{\mathcal{M}_{K_{n+1}}^{(0)}}\right) = n^n$. 
For $k = n-1$, 
$\mathcal{M}_{K_{n+1}}^{(n-1)} = \mathcal{M}_{K_{n+1}}$. The minimal 
free resolution of the $K_{n+1}$-parking function ideal 
$\mathcal{M}_{K_{n+1}}$ is the cellular resolution 
supported on the first barycentric subdivision $\mathbf{Bd}(\Delta_{n-1})$ of an
$n-1$-simplex $\Delta_{n-1}$ and 
$\dim_{\mathbb{K}}\left(\frac{R}{\mathcal{M}_{K_{n+1}}}\right) = 
|{\rm PF}(K_{n+1})|=|{\rm SPT}(K_{n+1})|=(n+1)^{n-1}$.
For $k=1$, the $1$-skeleton ideal $\mathcal{M}_{K_{n+1}}^{(1)}$ has 
a minimal cocellular resolution supported on the labelled polyhedral
complex induced by any generic arrangement of two tropical hyperplanes
in $\mathbb{R}^{n-1}$ (see Theorem 4.6 of \cite{D2})  and 
$\dim_{\mathbb{K}}\left(\frac{R}{\mathcal{M}_{K_{n+1}}^{(1)}}\right) = 
(2n-1)(n-1)^{n-1}$.

\noindent
{\bf{Betti numbers of  $\mathcal{M}_{K_{n+1}}^{(k)}$}} :
We now express the $k$-skeleton ideal 
$\mathcal{M}_{K_{n+1}}^{(k)}$ of $K_{n+1}$ as an Alexander dual
of a multipermutohedron ideal. Let $\mathbf{u}(\mathbf{m}) = 
(1,2,\ldots,k,k+1,\ldots,k+1) \in 
\mathbb{N}^{n}$, where $\mathbf{m}=(1,\ldots,1,n-k) \in 
\mathbf{\mathbb{N}}^{k+1}$.
For $k=0$, $\mathbf{u}(\mathbf{m}) = (1,\ldots,1) \in 
\mathbb{N}^{n}$, while  for $k=n-1$,
$\mathbf{u}(\mathbf{m}) = (1,2,\ldots,n) \in 
\mathbb{N}^{n}$. 
Let $I(\mathbf{u}(\mathbf{m}))^{[\mathbf{n}]}$ be the Alexander dual 
of the multipermutohedron ideal $I(\mathbf{u}(\mathbf{m}))$ with respect to 
$[\mathbf{n}] = (n,\ldots,n) \in \mathbf{N}^n$.
\begin{lemma}
 For $0 \le k \le n-1$, $\mathcal{M}_{K_{n+1}}^{(k)} = 
I(\mathbf{u}(\mathbf{m}))^{[\mathbf{n}]}$.
\label{Lem1}
\end{lemma}
\begin{proof}
 Using Proposition 5.23 of \cite{MS}, it follows from the Lemma 2.3 of \cite{AC1}. \hfill $\square$
\end{proof}

The multigraded Betti numbers of Alexander duals of a
 multipermutohedron ideal are described in terms of
dual $\mathbf{m}$-isolated subsets 
(see Definition 3.1 and Theorem 3.2 of \cite{AC1}). For the particular
case of $\mathbf{m}= (1,\ldots,1,n-k) \in \mathbb{N}^{k+1}$, the notion of 
dual $\mathbf{m}$-isolated subsets can be easily described. 
Let  $J =\{j_1,\ldots,j_t\} \subseteq
[n]$ be a non-empty subset with $0=j_0 <j_1 < \ldots < j_t$.
\begin{enumerate}
 \item  $J$ is a dual $\mathbf{m}$-{\em isolated subset of type}-1 
if $J \subseteq [k+1]$ and its dual weight
${\rm dwt}(J) = t-1$. Let
$\mathcal{I}_{\mathbf{m}}^{*,1}$
 be the set of dual $\mathbf{m}$-isolated subsets
of type-1 and let $\mathcal{I}_{\mathbf{m}}^{*,1} \langle i \rangle =
\{ J \in \mathcal{I}_{\mathbf{m}}^{*,1} : {\rm dwt}(J) = i \}$.
\item $J =\{j_1,\ldots,j_t\}$  is a dual $\mathbf{m}$-{\em isolated subset 
of type}-2 if $J\setminus \{j_t\} \subseteq [k], ~ k+1 < j_t \le n$ and its dual weight
${\rm dwt}(J) = (t-1)+(j_t - k)$. Let
$\mathcal{I}_{\mathbf{m}}^{*,2}$
 be the set of dual-$\mathbf{m}$ isolated subsets
of type-2 and let $\mathcal{I}_{\mathbf{m}}^{*,2} \langle i \rangle =
\{ J \in \mathcal{I}_{\mathbf{m}}^{*,2} : {\rm dwt}(J) = i \}$.
\end{enumerate}
Let $\mathcal{I}_{\mathbf{m}}^{*} = 
\mathcal{I}_{\mathbf{m}}^{*,1} \coprod \mathcal{I}_{\mathbf{m}}^{*,2}$
be the set of all dual $\mathbf{m}$-isolated subsets and 
$\mathcal{I}_{\mathbf{m}}^{*}\langle i \rangle = 
\mathcal{I}_{\mathbf{m}}^{*,1} \langle i \rangle 
 \coprod \mathcal{I}_{\mathbf{m}}^{*,2} \langle i \rangle$.

Consider $\lambda = (\lambda_1, \lambda_2, \ldots, \lambda_n) $ with 
$\lambda_i = \begin{cases} n-i+1 & {\rm if}~ 1 \le i \le k,\\
             n-k & {\rm if}~k+1 \le i \le n.
            \end{cases} $
Let $e_1,e_2,\ldots,e_n$ be the standard basis vectors of 
$\mathbb{R}^n$. For $0 \le i < j \le n$, 
we set $\varepsilon(i,j) = \sum_{l = i+1}^j e_l$. For any
 $J = \{j_1,\ldots,j_t\} \in \mathcal{I}_{\mathbf{m}}^{*}$,
let $\mathbf{b}(J) = \sum_{\alpha= 1}^t 
\lambda_{j_{\alpha}} ~ \varepsilon(j_{\alpha-1}, j_{\alpha}) \in \mathbb{N}^n$.

\begin{proposition}
 For $\mathbf{b}=(b_1,\ldots,b_n) \in \mathbb{N}^n$ and 
$ 1 \le i \le n-1$, let 
$\beta_{i-1,\mathbf{b}}\left(\mathcal{M}_{K_{n+1}}^{(k)}\right)$ 
be an ${(i-1)}^{th}$ 
multigraded Betti number of $\mathcal{M}_{K_{n+1}}^{(k)}$
 in degree $\mathbf{b}$. Then
the following statements hold.
\begin{enumerate}
 \item[{\rm (i)}] For $J=\{j_1,\ldots,j_t\} 
\in \mathcal{I}_{\mathbf{m}}^{*,1}\langle i-1 \rangle$, 
 $\beta_{i-1,\mathbf{b}(J)}\left(\mathcal{M}_{K_{n+1}}^{(k)}\right) = 1$,
 where $t=i$.
\item[{\rm (ii)}] For 
$J =\{ j_1,\ldots,j_t\} \in \mathcal{I}_{\mathbf{m}}^{*,2}\langle i-1 \rangle$, 
 $\beta_{i-1,\mathbf{b}(J)}\left(\mathcal{M}_{K_{n+1}}^{(k)}\right) = 
{j_t-j_{t-1}-1 \choose k-j_{t-1}}$, where $t+j_t-k=i$. 
\item[{\rm (iii)}] If $\mathbf{b}= \pi \mathbf{b}(J)$ is a
 permutation of $\mathbf{b}(J)$ for
some $J \in \mathcal{I}_{\mathbf{m}}^{*} \langle i-1 \rangle$ 
 and some $\pi \in \mathfrak{S}_n$,
then $\beta_{i-1,\mathbf{b}}\left(\mathcal{M}_{K_{n+1}}^{(k)}\right) = 
\beta_{i-1,\mathbf{b}(J)}\left(\mathcal{M}_{K_{n+1}}^{(k)}\right)$. Otherwise,
$\beta_{i-1,\mathbf{b}}\left(\mathcal{M}_{K_{n+1}}^{(k)}\right) = 0$.
\item[{\rm (iv)}] The ${(i-1)}^{th}$-Betti number 
$\beta_{i-1}\left(\mathcal{M}_{K_{n+1}}^{(k)}\right)$ of 
$\mathcal{M}_{K_{n+1}}^{(k)}$ is given by, 
\[
 \beta_{i-1}\left(\mathcal{M}_{K_{n+1}}^{(k)}\right) = 
\beta_{i}\left(\frac{R}{\mathcal{M}_{K_{n+1}}^{(k)}}\right) = 
\sum_{J \in \mathcal{I}_{\mathbf{m}}^{*,1}\langle i-1 \rangle} \beta_{i-1}^J +
\sum_{\tilde{J} \in \mathcal{I}_{\mathbf{m}}^{*,2} \langle i-1 \rangle}
 \beta_{i-1}^{\tilde{J}} ,
\]
where $\beta_{i-1}^J  = \prod_{\alpha=1}^i {j_{\alpha+1} \choose j_{\alpha}}$ 
for $J=\{j_1,\ldots,j_i\} \in 
\mathcal{I}_{\mathbf{m}}^{*,1} \langle i-1 \rangle$ and 
$\beta_{i-1}^{\tilde{J}}  = 
\left[\prod_{\alpha=1}^t {l_{\alpha+1} \choose l_{\alpha}}\right] 
{l_t-l_{t-1}-1 \choose k -l_{t-1}}$
 for $\tilde{J}=\{l_1,\ldots,l_t\} \in 
\mathcal{I}_{\mathbf{m}}^{*,2} \langle i-1 \rangle$. Here,
$j_{i+1}=l_{t+1}=n$.
\end{enumerate}
\label{Prop1}
\end{proposition}
\begin{proof}
 Since $\mathcal{M}_{K_{n+1}}^{(k)} = 
I(\mathbf{u}(\mathbf{m}))^{[\mathbf{n}]}$, it follows from Theorem 3.2 and 
Corollary 3.4 of \cite{AC1}. \hfill $\square$
\end{proof}
 Proposition \ref{Prop1} describes all multigraded Betti 
numbers of $\mathcal{M}_{K_{n+1}}^{(k)}$. We hope that
it could be helpful in constructing a concrete minimal 
resolution of  $\mathcal{M}_{K_{n+1}}^{(k)}$.

\begin{corollary}
 Assume that $n \ge 3$ and $1 \le i \le n-1$. Then 
$\beta_{i-1}\left(\mathcal{M}_{K_{n+1}}^{(1)}\right)
= i {n+1 \choose i+1}$ and 
\[
 \beta_{i-1}\left(\mathcal{M}_{K_{n+1}}^{(n-2)}\right) =
 \sum_{0 <j_1 <\ldots<j_i <n} 
\frac{n!}{j_1!(j_2-j_1)! \cdots (j_n-j_i)!} +
 \sum_{0 <l_1 < \ldots < l_{i-3} < n-1} 
\frac{n!(n-l_{i-3}-1)}{l_1! (l_2-l_1)! \cdots (n-l_{i-3})!}.
\]
\label{Cor1}
\end{corollary}
\begin{proof}
 For $k=1$, we have $\mathbf{m}=(1,n-1) \in \mathbb{N}^2$.
 We can easily see that
$\mathcal{I}_{\mathbf{m}}^{*}\langle i-1 \rangle = 
\{\{1,i\},\{i+1\}\}$ for $i \ge 2$ and
$\mathcal{I}_{\mathbf{m}}^{*}\langle 0 \rangle = \{\{1\},\{2\}\}$. Thus, 
$\beta_0(\mathcal{M}_{K_{n+1}}^{(1)}) = 
\beta_0^{\{1\}} + \beta_0^{\{2\}} = {n \choose 1}+{n \choose 2}=
{n+1 \choose 2}$. For $ i \ge 2$, 
\begin{eqnarray*}
 \beta_{i-1}(\mathcal{M}_{K_{n+1}}^{(1)}) & =  & \beta_{i-1}^{\{1,i\}} + 
\beta_{i-1}^{\{i+1\}} = {i \choose 1}{n \choose i} {i-2 \choose 0} 
+{n \choose i+1}{i \choose 1}\\
& = & i {n \choose i} + i {n \choose i+1} = i {n+1 \choose i+1},
\end{eqnarray*}
which is same as $\beta_i\left(\frac{R}{\mathcal{M}_{K_{n+1}}^{(1)}}\right)=
\sum_{j=1}^n j {j-1 \choose i-1}$ obtained in \cite{D2}.

For $k=n-2$, $J =\{j_1,\ldots,j_i\} \in 
\mathcal{I}_{\mathbf{m}}^{*,1}\langle i-1 \rangle$ if and 
only if $J \subseteq [n-1]$ and 
$\beta_{i-1}^J = \prod_{\alpha=1}^i {j_{\alpha+1} \choose j_{\alpha}}$. 
Also, $\tilde{J}=\{l_1,\ldots,l_t\} \in 
\mathcal{I}_{\mathbf{m}}^{*,2}\langle i-1 \rangle$ if and only if 
$l_{t-1} \le n-2,~ l_t=n$ and $t=i-2$. Since, 
$\beta_{i-1}^{\tilde{J}} = \left[\prod_{\alpha=1}^{i-3} 
{l_{\alpha+1} \choose l_{\alpha}} \right] 
{n-l_{i-3}-1 \choose n-l_{i-3}-2}$, we get
the desired expression for 
$\beta_{i-1}\left(\mathcal{M}_{K_{n+1}}^{(n-2)}\right)$. 
\hfill $\square$
\end{proof}

\noindent
{\bf{Standard monomials of  $\mathcal{M}_{K_{n+1}}^{(k)}$}} : 
A monomial $\mathbf{x}^{\mathbf{b}} = \prod_{j=1}^n x_j^{b_j}$ is called
a {\em standard monomial} of $\frac{R}{\mathcal{M}_{K_{n+1}}^{(k)}}$ (or 
$\mathcal{M}_{K_{n+1}}^{(k)}$) if 
$\mathbf{x}^{\mathbf{b}} \notin \mathcal{M}_{K_{n+1}}^{(k)}$.
We have seen that $I(\mathbf{u}(\mathbf{m}))^{[\mathbf{n}]} = 
\mathcal{M}_{K_{n+1}}^{(k)}$.
Thus the number of standard monomials of
 $\mathcal{M}_{K_{n+1}}^{(k)}$ is precisely
the number of $\lambda$-parking functions for
 $\lambda = (\lambda_1, \ldots, \lambda_n)$, where
$\lambda_i =n-i+1$ for $1 \le i \le k$ and $\lambda_j = n-k$ for $k+1 \le j \le n$ (see \cite{PoSh}).
\begin{definition}
{\rm  Let $\lambda = (\lambda_1,\ldots,
\lambda_n) \in \mathbb{N}^n$ with $\lambda_1 \ge \lambda_2 \ge \ldots \ge \lambda_n$. A finite sequence
$ \mathcal{P}=(p_1,\ldots,p_n) \in \mathbb{N}^n$ is called
a {\em $\lambda$-parking function} of length $n$, if 
a non-decreasing rearrangement
$p_{i_1} \le p_{i_2}
\le \ldots \le p_{i_n}$ of $\mathcal{P}$ satisfies 
$p_{i_j} < \lambda_{n-j+1}$ 
for $1 \le j \le n$. Let ${\rm PF}(\lambda)$ be the set 
of $\lambda$-parking functions. }
\end{definition}
Ordinary parking functions of length $n$ are precisely  
$\lambda$-parking functions
of length $n$ for $\lambda=(n,n-1,\ldots,2,1) \in \mathbb{N}^n$.
The number of $\lambda$-parking functions is given by the following
Steck determinant formula. Let $\Lambda (\lambda_1,\ldots,\lambda_n) = 
\left[\frac{\lambda_{n-i+1}^{j-i+1}}{(j-i+1)!}\right]_{1 \le i,j \le n}$. In other words,
the $(i,j)^{th}$ entry of the 
$n \times n$ matrix
$\Lambda (\lambda_1,\ldots,\lambda_n)$ is
$\frac{\lambda_{n-i+1}^{j-i+1}}{(j-i+1)!}$,
where, by convention,  $\frac{1}{(j-i+1)!} =0$ for $i > j+1$.
The determinant $\det(\Lambda(\lambda_1,\ldots,\lambda_n))$ 
is called a {\em Steck determinant}. 
 We have,
\[
 |{\rm PF}(\lambda) | = n!~~\det (\Lambda (\lambda_1,\ldots,\lambda_n)) = 
n!~~
\det \left[\frac{\lambda_{n-i+1}^{j-i+1}}{(j-i+1)!}\right]_{1 \le i,j \le n}.
\]
Thus, 
\[
 \dim_{\mathbb{K}}\left(\frac{R}{\mathcal{M}_{K_{n+1}}^{(k)}}\right) =
 n!~~\det( \Lambda(n,n-1,\ldots,n-k+1,n-k,\ldots,n-k)).
\]
We proceed to evaluate
Steck determinant and compute the number of standard monomials
of $\mathcal{M}_{K_{n+1}}^{(k)}$ as indicated in \cite{CK}.
For more on parking functions, we refer to \cite{PiSt, PoSh, Ya}. 

More generally, for $ a,b \ge 1$, we consider 
the complete multigraph $K_{n+1}^{a,b}$ on the vertex
set $V$ with adjacency matrix $A(K_{n+1}^{a,b}) =[a_{ij}]_{0\le i,j\le n}$ given by
$a_{0,i}=a_{i,0}=a$ and $a_{i,j} =b$ for $i,j \in V \setminus \{0\}$; $i\ne j$. In other
words, $K_{n+1}^{a,b}$ has exactly $a$ number of edges between the root $0$ and 
any other vertex $i$, while it has exactly $b$ number of edges between distinct
non-rooted vertices $i$ and $j$. Clearly, $K_{n+1}^{1,1} 
= K_{n+1}$.
The $k$-skeleton ideal 
$\mathcal{M}_{K_{n+1}^{a,b}}^{(k)}$
of $K_{n+1}^{a,b}$ is given by
\[
 \mathcal{M}_{K_{n+1}^{a,b}}^{(k)} = 
\left\langle \left(\prod_{j \in A} x_j \right)^{a+(n-|A|)b} : 
\emptyset \ne A \subseteq [n]; ~|A| \le k+1 \right\rangle .
\]
Let $\mathbf{u}^{a,b}(\mathbf{m}) = 
(a,a+b,\ldots,a+(k-1)b,a+kb,\ldots,a+kb) \in \mathbb{N}^n$. 
Then, as in Lemma \ref{Lem1}, we have
\[ 
I(\mathbf{u}^{a,b}(\mathbf{m}))^{[\mathbf{2a+(n-1)b-1}]} = 
\mathcal{M}_{K_{n+1}^{a,b}}^{(k)}; \quad \quad (0 \le k \le n-1),
\]
 where 
$\mathbf{2a+(n-1)b-1} = (2a+(n-1)b-1, \ldots, 2a+(n-1)b-1) \in \mathbb{N}^n$.
Thus, 
\begin{equation}
 \dim_{\mathbb{K}}\left(\frac{R}{{\mathcal{M}}_{K_{n+1}^{a,b}}^{(k)}}\right) =
 n!~~\det( \Lambda(\lambda_1^{a,b}, \ldots, \lambda_n^{a,b})),
\label{Eqn1}
\end{equation}
where
 $\lambda_i^{a,b} = a+(n-i)b$ for $1 \le i \le k$ and 
$\lambda_j^{a,b} = a+(n-k-1)b$ for $k+1 \le j \le n$.

Consider the polynomial $f_n(x) = 
\det (\Lambda(x+(n-1)b,x+(n-2)b,\ldots,x+b,x))$ in an indeterminate $x$. 
In other words, we have
\[ f_n(x) = \det \begin{bmatrix} \frac{x}{1!}& \frac{x^2}{2!} &                
\frac{x^3}{3!} & \ldots & \frac{x^{n-1}}{(n-1)!} &\frac{x^n}{n!}\\
1 & \frac{x+b}{1!} & \frac{(x+b)^2}{2!}& \ldots & 
\frac{(x+b)^{n-2}}{(n-2)!} & \frac{(x+b)^{n-1}}{(n-1)!}\\
0 & 1 & \frac{x+2b}{1!}& \ldots & \frac{(x+2b)^{n-3}}{(n-3)!} & 
\frac{(x+2b)^{n-2}}{(n-2)!}\\
\vdots & \vdots & \vdots & \ddots & \vdots & \vdots \\
0 & 0 & 0 & \ldots  & \frac{(x+(n-2)b)}{1!} & \frac{(x+(n-2)b)^2}{2!}\\
0 & 0 & 0 & \ldots & 1 & \frac{x+(n-1)b}{1!} 
\end{bmatrix}.
\]
Also, for $1 \le k \le n-2$, consider another polynomial $g_{n;k}(x)$ in $x$ 
given by
\[ g_{n;k}(x) = \det (\Lambda(x+kb, x+(k-1)b,\ldots,x+b,x,\ldots,x)),\]
where the last $n-k$ coordinates  in $ (x+kb, x+(k-1)b,\ldots,x+b,x,\ldots,x)$
are $x$.
\begin{proposition} The polynomials $f_n(x)$ and $g_{n;k}(x)$ are given as follows.
\begin{enumerate}
 \item [{\rm (1)}] $f_n(x) = \frac{x(x+nb)^{n-1}}{n!}$.
\item[{\rm (2)}] $g_{n;k}(x) = \sum_{j=0}^k \frac{1}{j!} \frac{x^{n-j}}{(n-j)!}
(k-j+1)(k+1)^{j-1}b^j$.
\end{enumerate}
\label{Prop2}
\end{proposition}
\begin{proof}
 We see that $f_1(x) = x$ and $f_2(x) = \frac{x(x+2b)}{2!}$. By induction on $n$, we assume
that $f_j(x) = \frac{x(x+jb)^{j-1}}{(j-1)!}$ for $ 1 \le j \le n-1$. Further, using
properties of determinants, we observe that the derivative $f_n^{\prime}(x)$ of 
$f_n(x)$ satisfies $f_n^{\prime}(x) = f_{n-1}(x+b)$. This shows that
$f_n^{\prime}(x) = \frac{(x+b)(x+nb)^{n-2}}{(n-1)!}$. As $f_n(0) = 0$, on integrating 
$f_n^{\prime}(x) = \frac{(x+b)(x+nb)^{n-2}}{(n-1)!}$ by parts, we get (1).

Again using properties of determinants, we see that 
the $(n-k-1)^{th}$ derivative $g_{n;k}^{(n-k-1)}(x)$ of $g_{n;k}(x)$ satisfies 
$g_{n;k}^{(n-k-1)}(x) = f_{k+1}(x) = \frac{x(x+(k+1)b)^k}{(k+1)!}
= \sum_{j=0}^k {k \choose j} x^{k-j+1} \frac{(k+1)^j b^j}{(k+1)!}$.
Since $g_{n;k}(0) = g_{n;k}^{\prime}(0) = \ldots = g_{n;k}^{(n-k-1)}(0) = 0$ and
the $(n-k-1)^{th}$ derivative of 
$\frac{x^{n-j}}{(n-j)(n-j-1)\ldots (k-j+2)}$ is $x^{k-j+1}$,
we get $g_{n;k}(x) = \sum_{j=0}^k {k \choose j} 
\frac{x^{n-j}}{(n-j) (n-j-1)\cdots(k-j+2)}
\frac{(k+1)^jb^j}{(k+1)!}$. This proves (2).
\hfill $\square$
\end{proof}
\begin{corollary} The number of standard monomials of 
$\frac{R}{\mathcal{M}_{K_{n+1}^{a,b}}}$ and 
$\frac{R}{\mathcal{M}_{K_{n+1}^{a,b}}^{(k)}}$ are given as follows.
\begin{enumerate}
 \item[{\rm (1)}] $ \dim_{\mathbb{K}}\left(\frac{R}{{\mathcal{M}}_{K_{n+1}^{a,b}}}\right)
= a(a+nb)^{n-1}$.
\item[{\rm (2)}]  $ \dim_{\mathbb{K}}\left(\frac{R}{{\mathcal{M}}_{K_{n+1}^{a,b}}^{(k)}}\right)
= \sum_{j=0}^k {n \choose j} (a+(n-k-1)b)^{n-j}(k-j+1)(k+1)^{j-1}b^j $.
\end{enumerate}
In particular, for $k=1$ and $k=n-2$, we have
$ \dim_{\mathbb{K}}\left(\frac{R}{{\mathcal{M}}_{K_{n+1}^{a,b}}^{(1)}}\right) =
\left(a+(n-2)b\right)^{n-1}(a+(2n-2)b)$ and 
$\dim_{\mathbb{K}}\left(\frac{R}{{\mathcal{M}}_{K_{n+1}^{a,b}}^{(n-2)}}\right)
= a(a+nb)^{n-1}+(n-1)^{n-1} b^{n}$.
\label{Cor2}
\end{corollary}
\begin{proof}
 From (\ref{Eqn1}), we have 
\[
\dim_{\mathbb{K}}\left(\frac{R}{{\mathcal{M}}_{K_{n+1}^{a,b}}}\right) =
 n!~f_n(a) \quad {\rm and}
\quad 
 \dim_{\mathbb{K}}\left(\frac{R}{{\mathcal{M}}_{K_{n+1}^{a,b}}^{(k)}}\right) =
 n! ~ g_{n;k}(a+(n-k-1)b).
\]
In view of Proposition \ref{Prop2}, we get (1) and (2). 

For $k=1$, we see that $g_{n;1}^{(n-2)}(x) =f_2(x) = \frac{x(x+2b)}{2!} = 
\frac{x^2}{2!} + bx$. As 
$g_{n;1}^{(j)}(0) = 0$ for $0 \le j \le n-2$, we obtain 
\[ g_{n;1}(x) = \frac{x^n}{n!} + \frac{bx^{n-1}}{(n-1)!} = \frac{x^{n-1}(x+nb)}{n!}.
\]
Now $\dim_{\mathbb{K}}\left(\frac{R}{{\mathcal{M}}_{K_{n+1}^{a,b}}^{(1)}}\right) = 
n! ~ g_{n;1}(a+(n-2)b)$ yields the desired result. 

Also, for $k=n-2$, we have
$ g_{n;n-2}^{\prime}(x) = f_{n-1}(x) = \frac{x(x+(n-1)b)^{n-2}}{(n-1)!}$. On integrating 
it by parts, we get
$g_{n;n-2}(x) = \frac{x(x+(n-1)b)^{n-1}}{(n-1)! (n-1)} - \frac{(x+(n-1)b)^n}{n! (n-1)} +C$, where
$C$ is a constant of integration. Since $g_{n;n-2}(0)=0$, we get $C = \frac{(n-1)^{n-1}b^n}{n!}$.
Hence, $g_{n;n-2}(x) = \frac{1}{n!} [(x-b) (x+(n-1)b)^{n-1} + (n-1)^{n-1}b^n]$.
Again, from 
$\dim_{\mathbb{K}}\left(\frac{R}{{\mathcal{M}}_{K_{n+1}^{a,b}}^{(n-2)}}\right) = 
n! ~ g_{n;n-2}(a+b)$, we get the desired result.
\hfill $\square$
\end{proof}
\begin{remarks}
 {\rm 
\begin{enumerate}
 \item[{\rm (1)}] It can be easily checked that 
 the determinant $\det(Q_{K_{n+1}^{a,b}})$ of 
the truncated signless Laplace matrix 
$Q_{K_{n+1}^{a,b}}$
of $K_{n+1}^{a,b}$ satisfies
\[
 \dim_{\mathbb{K}}\left(\frac{R}{{\mathcal{M}}_{K_{n+1}^{a,b}}^{(1)}}\right) =
\left(a+(n-2)b\right)^{n-1}(a+(2n-2)b) =  \det(Q_{K_{n+1}^{a,b}}).
\]
This extends Corollary 3.4 of \cite{D2} to the 
complete multigraph $K_{n+1}^{a,b}$.
\item[{\rm (2)}] We have $g_{n;n-2}^{\prime}(x) = f_{n-1}(x) =\frac{x(x+(n-1)b)^{n-2}}{(n-1)!}=
\sum_{j=0}^{n-2} {n-2 \choose j} \frac{x^{n-1-j}}{(n-1)!} (n-1)^j b^j $.  Thus on integrating
$g_{n;n-2}^{\prime}(x)$ in two ways, we get a polynomial identity 
\[
 g_{n;n-2}(x) = \frac{(x-b)(x+(n-1)b)^{n-1}+(n-1)^{n-1}b^n}{n!} = 
\frac{\sum_{j=0}^{n-2} {n \choose j} x^{n-j} (n-j-1) (n-1)^{j-1}b^j}{n!}.
\]
On substituting $x=a+b$, we get an identity
\[
 \sum_{j=0}^{n-2} {n \choose j} (a+b)^{n-j} (n-j-1) (n-1)^{j-1}b^j = a(a+nb)^{n-1}+(n-1)^{n-1}b^n
\]
for positive integers $a$ and $b$. Taking $a=b=1$, it
justifies the equality
\[ 
\sum_{j=0}^{n-2} {n \choose j} 2^{n-j} (n-j-1) (n-1)^{j-1} = (n+1)^{n-1} + (n-1)^{n-1}
\]
 described in the Remark 3.3 of \cite{D1}.
\item[{\rm (3)}] From Corollary \ref{Cor2}, the 
number of spherical $K_{n+1}^{a,b}$-parking function
is given by $|{\rm sPF}(K_{n+1}^{a,b})| = (n-1)^{n-1}b^n$. Note that this number is independent of $a$.
\end{enumerate}
\label{Rem1}
}
\end{remarks}

\noindent
{\bf{DFS burning Algorithm}} : 
We shall briefly describe Depth-First-Search (DFS) burning algorithms
of Perkinsons-Yang-Yu \cite{PYY} and Gaydarov-Hopkins \cite{GH}.
Firstly, we set up graph theoretic notations and invariants
needed for the DFS
algorithm. Let $G$ be a connected graph on the vertex set 
$V(G) = V=\{0,1,\ldots,n\}$. Suppose $A(G) = [a_{ij}]_{0 \le i,j \le n}$ 
is the (symmetric) adjacency matrix of $G$. Since $G$ has no loops, $a_{ii}=0$.
Let $E(i,j)=E(j,i)$ be the set of edges between $i$ and $j$ for distinct 
$i,j \in V$. If $E(i,j) \ne \emptyset$, then $i$ and $j$ are called 
{\em adjacent vertices} and we write $ i \sim j$. On the other hand,
if $i$ and $j$ are non-adjacent, we write $i \nsim j$.
We have $|E(i,j)|=a_{ij}$. The graph $G$ is called a {\em simple graph} if
$|E(i,j)|=a_{ij} \le 1$ for $i,j \in V$. Otherwise, $G$ is called a 
{\em multigraph}. The set $E(G) = \bigcup_{i,j \in V} E(i,j)$ is the set 
of edges of the graph $G$. If $v \in V$, 
then $G - \{v\}$ denotes
the graph on the vertex set $V \setminus \{v\}$ obtained from $G$ on deleting 
the vertex $v$ and all the edges through $v$. 
If $e\in E(G)$ is an edge of $G$, then
$G - \{e\}$ denotes the graph on the 
vertex set $V$ obtained from $G$
on deleting the edge $e$. 
If $E(i,j) \ne \emptyset$, then $G - E(i,j)$ 
denotes the 
graph on vertex set $V$ obtained 
from $G$ on deleting all the edges between $i$
and $j$. Fix a root $r \in V$ of $G$ (usually, we take $r=0$). 
Set $\widetilde{V} = V \setminus \{r \}$. Let ${\rm SPT}(G)$ be the
set of  spanning trees of $G$ rooted at $r$. We orient spanning tree 
$T \in {\rm SPT}(G)$ so that all paths in $T$ are directed away from the 
root $r$. For every $j \in \widetilde{V}$, there is a unique oriented 
path in $T$ from the root $r$ to $j$. An 
$i \in \widetilde{V}$ lying on this 
unique path in $T$ is called an {\em ancestor} of $j$ in $T$. 
Equivalently, we say that $j$ is a {\em descendent} 
of $i$ in $T$. If in addition,
$i$ and $j$ are adjacent in $T$, then 
we say that $i$ is a {\em parent} of its {\em child}
$j$. Every child $j$ has a unique parent ${\rm par}_T(j)$ in $T$.
By an {\em inversion} of $T \in {\rm SPT}(G)$, we mean an ordered pair
$(i,j)$ of vertices such that $i$ is an ancestor of $j$ in $T$ with $ i > j$. 
The total number of inversions of a spanning tree $T$ is denoted by
${\rm inv}(T)$.
An inversion $(i,j)$ of $T$ is called a {\em $\kappa$-inversion} of $T$ if
$i$ is not the root $r$ and ${\rm par}_T(i)$ is adjacent to $j$ in $G$. The
{\em $\kappa$-number} $\kappa(G, T)$ of $T$ in $G$ is given by
\[
 \kappa(G, T) = \sum_{\substack{i,j \in \widetilde{V}; \\ i > j}} 
 |E({\rm par}_T(i),j)|.
\]
For a simple graph $G$, $\kappa(G, T)$ is the total number of $\kappa$-inversions of
$T$. If $G = K_{n+1}$ with root $0$. then 
$\kappa(K_{n+1}, T) = {\rm inv}(T)$ for 
every $T \in {\rm SPT}(K_{n+1})$.
The invariant 
\[
g(G) = |E(G)|-|V(G)|+1 = 
\left(\sum_{0 \le j < i \le n} a_{ji} \right) -(n+1)+1 = 
\left(\sum_{0 \le j < i \le n} a_{ji} \right) - n
\]
is called the {\em genus} of the graph $G$. For a $G$-parking function 
$\mathcal{P} : \widetilde{V} \longrightarrow \mathbb{N}$, the {\em reverse sum}
of $\mathcal{P}$ is given by
\[{\rm rsum}(\mathcal{P}) = g(G) - {\rm sum}(\mathcal{P}) = 
 g(G) - \sum_{i \in \widetilde{V}} \mathcal{P}(i).
\]
In the definition of $G$-parking function, we have taken root $r=0$. 
For a root $r$ different from $0$, a notion 
of $G$-parking functions (with repect to
root $r$) $\mathcal{P}: V \setminus \{r\} 
\longrightarrow \mathbb{N}$
 can be easily defined (see \cite{PYY}). We are now in a position to 
describe DFS burning algorithm.

Let $G$ be a simple graph with a root $r \in V$. Applied to 
an input function $\mathcal{P}: V \setminus \{r\} 
\longrightarrow \mathbb{N}$,
the DFS algorithm of Perkinsons-Yang-Yu \cite{PYY} gives a subset
${\tt burnt\_\tt vertices}$ of burnt vertices and a subset
${\tt tree\_edges}$ of tree edges as an output. 
We imagine that a fire starts at
the root $r$ and spread to other vertices of $G$ 
according to the depth-first rule.
The value $\mathcal{P}(j)$ of the input function 
$\mathcal{P}$ can be considered as 
the number of water droplets available at 
vertex $j$ that prevents spread of
fire to $j$. If $i$ is a burnt vertex, 
then consider the largest non-burnt
vertex $j$ adjacent to $i$. If $\mathcal{P}(j) = 0$, 
then fire from $i$
will spread to $j$. In this case, add $j$ in  
${\tt burnt\_\tt vertices}$ and include the edge $(i,j)$ in 
${\tt tree\_edges}$. Now the fire spreads from the burnt vertex $j$.
On the other hand, if $\mathcal{P}(j) > 0$, then one water droplet available at $j$
will be used to prevent fire  from reaching $j$ through the edge $(i,j)$. In this case,
the dampened edge $(i,j)$ is removed from $G$, number of water droplets available at
$j$ is reduced to $\mathcal{P}(j) -1$ and the fire continue to 
spread from the burnt vertex $i$ through non-dampened edges. If all the 
edges from $i$ to unburnt vertices get dampened, then the search backtracks.
At the start, ${\tt burnt\_vertices} = \{r\}$ and ${\tt tree\_edges} = \{ \}$.

Perkinson, Yang and Yu \cite{PYY} constructed a bijection 
$\phi : {\rm PF}(G) \longrightarrow {\rm SPT}(G)$ using their DFS algorithm.
We state their result for future reference.

\begin{theorem}[Perkinson-Yang-Yu]
Let $G$ be a simple graph on $V$ with root $r$. If on applying 
DFS burning algorithm  to $\mathcal{P}: V \setminus \{r\} \longrightarrow 
\mathbb{N}$, the subset ${\tt burnt\_vertices}$ of burnt vertices is
$V$, then $\mathcal{P}$ is a $G$-parking function and  
the ${\tt tree\_edges}$ of tree edges form a spanning tree
$\phi(\mathcal{P})$ of $G$. Further, the mapping 
$\mathcal{P} \mapsto \phi(\mathcal{P})$ given by DFS algorithm induces
a bijection $\phi : {\rm PF}(G) \longrightarrow {\rm SPT}(G)$ such that
\[
 {\rm rsum}(\mathcal{P}) = g(G) - {\rm sum}(\mathcal{P}) = 
\kappa(G,\phi(\mathcal{P})).
\]
\label{Thm1}
\end{theorem}

Let $\sum_{\mathcal{P} \in {\rm PF}(G)} q^{{\rm rsum}(\mathcal{P})} $ be the 
{\em reversed sum 
enumerator} for $G$-parking functions. Then 
Theorem \ref{Thm1} establishes the following identity  
\[
 \sum_{\mathcal{P} \in {\rm PF}(G)} q^{{\rm rsum}(\mathcal{P})}
=  \sum_{T \in {\rm SPT}(G)} q^{\kappa(G,T)}.
\]
For $G = K_{n+1}$ with root $0$, 
${\rm PF}(K_{n+1}) = {\rm PF}(n)$ is the set of ordinary parking functions of length $n$ and
the above identity reduces to the identity
\[ \sum_{\mathcal{P} \in {\rm PF}(n)} q^{{\rm rsum}(\mathcal{P})}
=  \sum_{T \in {\rm SPT}(K_{n+1})} q^{{\rm inv}(T)}
\]
proved by Kreweras \cite{K}.

Let
${\rm sPF}(K_{n+1})$ be the set of spherical $K_{n+1}$-parking  functions and
 $\mathcal{U}_n$ be the set of uprooted trees on $[n]$.
As an application of Theorem \ref{Thm1}, we construct a bijection 
$\phi_n : {\rm sPF}(K_{n+1}) 
\longrightarrow \mathcal{U}_n$ and solve a conjecture of 
Dochtermann \cite{D2}. For $\mathcal{P} \in {\rm sPF}(K_{n+1})$, let 
$\widetilde{\mathcal{P}} : V \setminus \{0\}
\longrightarrow \mathbb{N}$ be given by 
$\widetilde{\mathcal{P}}(i) = \mathcal{P}(i) - 1$.
Equivalently, $\mathbf{x}^{\widetilde{\mathcal{P}}} = 
\frac{\mathbf{x}^{\mathcal{P}}}{m_{[n]}}$, where
$m_{[n]}=x_1 \cdots x_n$ is the generator of 
$\mathcal{M}_{K_{n+1}}$ corresponding to $[n]$. 
We say that $\widetilde{\mathcal{P}}$ is the 
{\em reduced spherical $K_{n+1}$-parking
function} associated to $\mathcal{P} \in {\rm sPF}(K_{n+1})$
and $\widetilde{{\rm sPF}}(K_{n+1}) = \{ \widetilde{\mathcal{P}} :
\mathcal{P} \in {\rm sPF}(K_{n+1}) \}$. Clearly,  
$\widetilde{{\rm sPF}}(K_{n+1}) \subseteq {\rm PF}(K_{n+1})$.
Let $K_n = K_{n+1} - \{0\}$ be the complete graph on the 
vertex set $V\setminus \{0\} = [n]$.

\begin{theorem}
 There exists a bijection $\phi_n: {\rm sPF}(K_{n+1}) \longrightarrow 
\mathcal{U}_n$ such that
\[
 {\rm sum}(\mathcal{P}) = {n \choose 2} - \kappa(K_n, \phi_n(\mathcal{P})) +1 , 
\quad \forall ~\mathcal{P} \in {\rm sPF}(K_{n+1}).
\]
 \label{Thm2}
\end{theorem}

\begin{proof}
 Let $\mathcal{P} \in {\rm sPF}(K_{n+1})$. Then 
$\widetilde{\mathcal{P}} \in {\rm PF}(K_{n+1})$. Choose
the largest vertex $r$ of $K_n = K_{n+1} - \{0\}$ such that 
$\widetilde{\mathcal{P}}(r) = 0$. We claim that 
$\widetilde{\mathcal{P}}(j) = 0 $ for some $j < r$. Otherwise, 
$\mathcal{P}(i) \ge 2, ~\forall~ i \in [n] \setminus \{r\}$, 
a contradiction to $\mathcal{P} \in {\rm sPF}(K_{n+1})$. Now consider $r$
to be the root of the complete graph $K_n$ on $[n]$. Then
$\widehat{\mathcal{P}}   
= \widetilde{\mathcal{P}} \mid_{[n] \setminus \{r\}}$ 
is a $K_n$-parking function. On applying Theorem \ref{Thm1},
we get a spanning tree $\phi(\widehat{P})$ of the complete graph $K_n$ with root $r$.
Since $\widehat{P}(i) \ge 1$ for $i >r$, all the edges $(r,i)$ are dampened. Hence, 
$\phi(\widehat{P})$ is a uprooted tree on $[n]$ with root $r$. Define 
$\phi_n(\mathcal{P}) = \phi(\widehat{P})$. Clearly, 
$\phi_n: {\rm sPF}(K_{n+1}) \longrightarrow 
\mathcal{U}_n$ is injective. As $|{\rm sPF}(K_{n+1})|= 
 |\mathcal{U}_n|=(n-1)^{n-1}$, it follows that $\phi_n$ is a bijection.
Also,
\[
 {\rm rsum}(\widehat{P}) = g(K_n) - \sum_{i \in [n] \setminus \{r\}} 
\widetilde{\mathcal{P}}(i) = 
\kappa(K_n, \phi(\widehat{P})).
\]
As $g(K_n) = {n \choose 2}-n+1, \widetilde{\mathcal{P}}(i) = 
\mathcal{P}(i) -1$ and $\widetilde{\mathcal{P}}(r) = 0$, we have ${\rm rsum}(\widehat{P}) = 
{n \choose 2}-{\rm sum}(\mathcal{P}) + 1$.   \hfill $\square$
\end{proof}

We now describe the DFS burning algorithm of
 Gaydarov-Hopkins \cite{GH} for multigraphs.
Consider a connected multigraph $G$ on $V= \{0,1,\ldots,n\}$ with root $r$. 
Let $E(i,j)=E(j,i)$ be the  
set of edges between distinct vertices $i$ and $j$.  Fix a total order on $E(i,j)$ 
for all distinct pair $\{i,j\}$ of vertices and write
$E(i,j) = \{ e_{ij}^0, e_{ij}^1, \ldots, e_{ij}^{a_{ij}-1}\}$, 
where $|E(i,j)|=a_{ij}$.  
Thus we assume that edges of the multigraph $G$ are labelled.
Applied to 
an input function $\mathcal{P}: V \setminus \{r\} 
\longrightarrow \mathbb{N}$,
the DFS algorithm for multigraphs gives a subset
${\tt burnt\_\tt vertices}$ of burnt vertices and a subset
${\tt tree\_edges}$ of tree edges with nonnegative labels on them as an output. 
As in the case of DFS algorithm for simple graphs,
we imagine that a fire starts at
the root $r$ and spread to other vertices of $G$ 
according to the depth-first rule.
 If $i$ is a burnt vertex, 
then consider the largest non-burnt
vertex $j$ adjacent to $i$. If $\mathcal{P}(j)  < a_{ij}=|E(i,j)|$,
then $\mathcal{P}(j)$ edges with higher labels, namely $e_{ij}^{a_{ij}-1},
\ldots, e_{ij}^{a_{ij}-\mathcal{P}(j)}$ will get dampened,
the edge $e_{ij}^{a_{ij}-\mathcal{P}(j)-1}$ 
with label $a_{ij}-\mathcal{P}(j)-1$ 
will be added to  
${\tt tree\_edges}$ and $j$ in included in  
${\tt burnt\_\tt vertices}$.
 Now fire will spread from the burnt vertex $j$.
On the other hand, if $\mathcal{P}(j) \ge a_{ij}$, then 
all the edges in $E(i,j)$ get
 dampened and $\mathcal{P}(j)$ reduced to $\mathcal{P}(j) -a_{ij}$.  
The fire continue to 
spread from the burnt vertex $i$ through non-dampened edges. If all the 
edges from $i$ to unburnt vertices get dampened, then the search backtracks.
At the start, ${\tt burnt\_vertices} = \{r\}$ and ${\tt tree\_edges} = \{ \}$.
Gaydarov and Hopkins \cite{GH} extended Theorem \ref{Thm1} to multigraphs using 
the DFS burning algorithm for multigraph. We state their result without proof.

\begin{theorem}[Gaydarov-Hopkins]
Let $G$ be a  multigraph on $V$ with root $r$. If on applying 
DFS burning algorithm  to $\mathcal{P}: V \setminus \{r\} \longrightarrow 
\mathbb{N}$, the subset ${\tt burnt\_vertices}$ of burnt vertices is
$V$, then $\mathcal{P}$ is a $G$-parking function and  
the ${\tt tree\_edges}$ of tree edges with labels form a spanning tree
$\phi(\mathcal{P})$ of $G$. Suppose $\ell(e)$ is the label on 
an edge $e$ of $\phi(\mathcal{P})$. Then the mapping 
$\mathcal{P} \mapsto \phi(\mathcal{P})$ given by DFS burning algorithm induces
a bijection $\phi : {\rm PF}(G) \longrightarrow {\rm SPT}(G)$ such that
\[
 {\rm rsum}(\mathcal{P}) =  
\kappa(G,T) + \sum_{e \in E(T)} \ell(e), \quad {\rm where}\quad T=\phi(\mathcal{P}).
\]
\label{Thm3}
\end{theorem}

Consider the complete multigraph $K_{n+1}^{a,b}$ on $V$. Let
${\rm sPF}(K_{n+1}^{a,b})$ be the set of spherical $K_{n+1}^{a,b}$-parking  functions. 
Let
 $\mathcal{U}_n^b$ be the set of uprooted tree $T$ on $[n]$
with label $\ell : E(T) \longrightarrow \{0,1,\ldots, b-1\}$ on the edges of
$T$ and 
a weight $\omega(r) \in \{0,1,\ldots,b-1\}$ assigned to .the root $r$ of $T$.
Clearly, $|\mathcal{U}_n^b|=b^n |\mathcal{U}_n| = b^n(n-1)^{n-1}$.
As an application of Theorem \ref{Thm3}, we construct a bijection 
\[
\phi_n^b : {\rm sPF}(K_{n+1}^{a,b}) 
\longrightarrow \mathcal{U}_n^b
\] 
extending Theorem \ref{Thm2}.
Let $m_{[n]}=(x_1 \cdots x_n)^a$ be the generator of 
$\mathcal{M}_{K_{n+1}^{a,b}}$ corresponding to $[n]$. 
 The 
{\em reduced spherical $K_{n+1}^{a,b}$-parking
function} $\widetilde{\mathcal{P}}$ associated to $\mathcal{P} \in {\rm sPF}(K_{n+1}^{a,b})$
is given by $\mathbf{x}^{\widetilde{\mathcal{P}}} = 
\frac{\mathbf{x}^{\mathcal{P}}}{m_{[n]}}$. In other words,
 $\widetilde{\mathcal{P}}(i) = \mathcal{P}(i) - a, ~\forall ~i \in [n]$.
Let  $\widetilde{{\rm sPF}}(K_{n+1}^{a,b}) = \{ \widetilde{\mathcal{P}} :
\mathcal{P} \in {\rm sPF}(K_{n+1}^{a,b}) \}$. Clearly,  
$\widetilde{{\rm sPF}}(K_{n+1}^{a,b}) \subseteq {\rm PF}(K_{n+1}^{a,b})$.
Let $K_n^b = K_{n+1}^{a,b} - \{0\}$ be the complete multigraph on the 
vertex set $V\setminus \{0\} = [n]$ such that $|E(i,j)|=b$ for every pair 
$\{i,j\}$ of vertices.

\begin{theorem}
 There exists a bijection $\phi_n^b: {\rm sPF}(K_{n+1}^{a,b}) \longrightarrow 
\mathcal{U}_n^b$ such that
\[
 {\rm rsum}(\mathcal{P}) +\omega(r)+1 =  \kappa(K_n, T) + \sum_{e \in E(T)} \ell(e) , 
\quad \forall ~\mathcal{P} \in {\rm sPF}(K_{n+1}),
\]
where $T=\phi_n^b(\mathcal{P})$ and $r$ is the root of $T$. 
 \label{Thm4}
\end{theorem}

\begin{proof}
 Let $\mathcal{P} \in {\rm sPF}(K_{n+1}^{a,b})$. Then 
$\widetilde{\mathcal{P}} \in {\rm PF}(K_{n+1}^{a,b})$. Choose
the largest vertex $r$ of $K_n^b = K_{n+1}^{a,b} 
\setminus \{0\}$ such that 
$\widetilde{\mathcal{P}}(r) < b$. We claim that 
$\widetilde{\mathcal{P}}(j) <b $ for some $j < r$. Otherwise, 
$\mathcal{P}(i) \ge a+b, ~\forall~ i \in [n] \setminus \{r\}$, 
a contradiction to $\mathcal{P} \in {\rm sPF}(K_{n+1}^{a,b})$. 
Now consider $r$
to be the root of the complete multigraph $K_n^b$ on $[n]$. Then
$\widehat{\mathcal{P}}   
= \widetilde{\mathcal{P}} \mid_{[n] \setminus \{r\}}$ 
is a $K_n^b$-parking function. 
On applying the DFS algorithm (Theorem \ref{Thm3}),
we get  $\phi(\widehat{P}) \in \mathcal{U}_n^b$ with root $r$. 
Set weight $\omega(r) = 
\widetilde{\mathcal{P}}(r)$.
The mapping $\phi_n^b: {\rm sPF}(K_{n+1}^{a,b}) \longrightarrow 
\mathcal{U}_n^b$ given by 
$\phi_n^b(\mathcal{P}) = \phi(\widehat{P})$ is clearly 
 injective. As
 \[
|{\rm sPF}(K_{n+1}^{a,b})|= 
 |\mathcal{U}_n^b|=b^n(n-1)^{n-1},
\]
 it follows that $\phi_n^b$ is a bijection.
Also,
\[
 {\rm rsum}(\widehat{P}) = g(K_n^b) - \sum_{i \in [n] \setminus \{r\}} 
\widetilde{\mathcal{P}}(i) = 
\kappa(K_n, \phi(\widehat{P})) + 
\sum_{e \in E(\phi(\widehat{P}))} \ell(e).
\]
Since ${\rm rsum}(\mathcal{P}) = 
g(K_{n+1}^{a,b}) -\sum_{i \in [n]} \mathcal{P}(i)$, we
verify that ${\rm rsum}(\widehat{P}) = 
{\rm rsum}(\mathcal{P}) +\omega(r) +1$.
 \hfill $\square$
\end{proof}
Let $G$ be a connected simple graph on $V$ with root $0$. 
Let $m_{[n]}$ be the generator of 
$\mathcal{M}_G$ corresponding to $[n]$. Then to
each spherical $G$-parking function $\mathcal{P}$,
we associate a 
{\em reduced spherical $G$-parking function} 
$\widetilde{\mathcal{P}}$ by 
$\mathbf{x}^{\widetilde{\mathcal{P}}} = 
\frac{\mathbf{x}^{\mathcal{P}}}{m_{[n]}}$. 
Let $\widetilde{{\rm sPF}}(G) = 
\{ \widetilde{\mathcal{P}} : \mathcal{P} \in {\rm sPF}(G) \}$
and $\mathcal{U}_{G - \{0\}}$ be the set of all uprooted spanning trees of
$G - \{0\}$.
\begin{corollary}
 Let $G$ be a connected simple graph on $V$ with a root $0$ such that
$\widetilde{{\rm sPF}}(G) \subseteq {\rm PF}(G)$. Then there
exists an injective map $\phi_G : {\rm sPF}(G) \longrightarrow \mathcal{U}_{G - \{0\}}$
induced by DFS-algorithm.
\label{Cor3} 
\end{corollary}
\begin{proof}
 Proceed as in Theorem \ref{Thm2}. \hfill $\square$
\end{proof}
In the next section, we shall see that $\phi_G$ need not be surjective.

\section{Spherical parking functions}
Let $G$ be a connected graph on $V=\{0,1,\ldots,n\}$ with root $0$.
As stated in the Introduction, 
$\mathcal{P} : \widetilde{V}=V \setminus \{0\} \longrightarrow 
\mathbb{N}$ is a spherical $G$-parking function if 
$\mathbf{x}^{\mathcal{P}} = 
\prod_{i \in [n]} x_i^{\mathcal{P}(i)} \in \mathcal{M}_G \setminus 
\mathcal{M}_G^{(n-2)}$. Let ${\rm PF}(G)$ (or ${\rm sPF}(G)$) be the set of 
$G$-parking functions (respectively, 
spherical $G$-parking functions). The 
$G$-parking function ideal $\mathcal{M}_G$ and 
its $(n-2)^{th}$ skeleton 
$\mathcal{M}_G^{(n-2)}$ can be defined even 
for disconnected graphs. Note that
if $G$ is connected but $G - \{0\}$ is disconnected, then 
${\rm PF}(G) \ne \emptyset$ although ${\rm sPF}(G) = \emptyset $.

Let $e_0$ be an edge of $G$ joining the root $0$ 
to another vertex. We shall compare
${\rm sPF}(G)$ with ${\rm sPF}(G^{\prime})$, 
where $G^{\prime}=  G - \{e_0\}$.
After renumbering vertices, we may assume 
that $e_0=e_{0,n}$ is an edge joining the root $0$
with $n$.
\begin{lemma}
 Let $G$ be a connected graph on $V$ and $G^{\prime} =
 G - \{e_0\}$. Then 
\[
 \mathcal{M}_{G^{\prime}} = ( \mathcal{M}_G : x_n ) = 
\{ z \in R : z\cdot x_n \in \mathcal{M}_G \}.
\]
Further, the multiplication map $\mu_{x_n} :
\{ \mathbf{x}^{\mathcal{P}} : \mathcal{P} \in {\rm sPF}(G^{\prime})\}
\longrightarrow \{ \mathbf{x}^{\mathcal{P}} : \mathcal{P} \in {\rm sPF}(G)\}$ induced by $x_n$ is a bijection. In particular,
$|{\rm sPF}(G) | = |{\rm sPF}(G^{\prime})|$.
\label{Lem2}
\end{lemma}
\begin{proof}
 For $\emptyset \ne  A \subseteq [n]$, let $m_A$ and $m_A^{\prime}$ be
 the generators of $\mathcal{M}_G$ and 
$\mathcal{M}_{G^{\prime}}$, respectively.  Clearly, 
$m_A = m_A^{\prime}$ if $n \notin A$ and
$m_A = m_A^{\prime} x_n$ if $n \in A$. This shows that 
$\mathcal{M}_{G^{\prime}} = ( \mathcal{M}_G : x_n )$. Also,
$\mathcal{M}_{G^{\prime}}^{(n-2)} = ( \mathcal{M}_G^{(n-2)} : x_n )$. 
Thus the natural sequences
of $R$-modules (or $\mathbb{K}$-vectors spaces)
\[
0 \rightarrow \frac{R}{\mathcal{M}_{G^{\prime}}} 
\stackrel{\mu_{x_n}}{\rightarrow} 
\frac{R}{\mathcal{M}_G} \rightarrow 
\frac{R}{\langle \mathcal{M}_G, ~x_n \rangle} \rightarrow 0 
\quad{\rm and}\quad
 0 \rightarrow \frac{R}{\mathcal{M}_{G^{\prime}}^{(n-2)}} 
\stackrel{\mu_{x_n}}{\rightarrow} 
\frac{R}{\mathcal{M}_G^{(n-2)}} \rightarrow 
\frac{R}{\langle \mathcal{M}_G^{(n-2)}, ~x_n \rangle} \rightarrow 0 
\]
 are short exact. Let $\alpha : \frac{R}{\mathcal{M}_{G^{\prime}}^{(n-2)}}
\rightarrow \frac{R}{\mathcal{M}_{G^{\prime}}}$ and 
$\beta : \frac{R}{\mathcal{M}_G^{(n-2)}}
\rightarrow \frac{R}{\mathcal{M}_G}$ be the natural projections.
Since $\langle \mathcal{M}_G, ~x_n \rangle = 
\langle \mathcal{M}_G^{(n-2)}, ~x_n \rangle $, the multiplication map 
$\mu_{x_n}$ induces an isomorphism 
${\rm ker}(\alpha)
\stackrel{\sim}{\rightarrow} {\rm ker}(\beta)$ 
between kernels ${\rm ker}(\alpha)$ and ${\rm ker}(\beta)$.
Also $\{ \mathbf{x}^{\mathcal{P}} : \mathcal{P} \in {\rm sPF}(G^{\prime})\}$ and $\{ \mathbf{x}^{\mathcal{P}} : \mathcal{P} \in {\rm sPF}(G)\}$ are 
monomial basis of ${\rm ker}(\alpha)$
and ${\rm ker}(\beta)$, respectively. Thus $\mu_{x_n}$ induces 
a bijection between the bases. 
\hfill $\square$  
\end{proof}
We now give a few applications of the  Lemma \ref{Lem2}.
\begin{proposition}
 Let $E$ be the set of all edges  of
 $K_{n+1}$ or $K_{n+1}^{a,b}$
 through the root $0$. Then
\begin{enumerate}
 \item[{\rm (1)}] $|{\rm sPF}(K_{n+1} - E)| 
 = |{\rm sPF}(K_{n+1})|$.
\item[{\rm (2)}] $|{\rm sPF}(K_{n+1}^{a,b} - E)| 
= |{\rm sPF}(K_{n+1}^{a,b})|$.
\item[{\rm (3)}] $|{\rm sPF}(K_{n+1}^{a,b})| = b^n (n-1)^{n-1}$.
\end{enumerate}
\label{Prop3} 
\end{proposition}
\begin{proof}
 By Lemma \ref{Lem2}, we know that the number of spherical $G$-parking functions and the number of spherical $(G - \{e_0\})$-parking
 functions
 are the same for any edge $e_0$ of $G$ through
the root $0$. 
Now, repeatedly applying Lemma \ref{Lem2}, we see that (1) and (2) hold.

Let $\mathbf{u}^{b}(\mathbf{m}) = 
(2b,3b,\ldots,nb,nb) \in \mathbb{N}^n$. Then 
as described in the Section 2, the 
Alexander dual $I(\mathbf{u}^{b}(\mathbf{m}))^{[\mathbf{(n+1)b-1}]}$ of 
the multipermutohedron ideal $I(\mathbf{u}^{b}(\mathbf{m}))$ with respect to 
$\mathbf{(n+1)b-1} = ((n+1)b-1, \ldots,(n+1)b-1) \in \mathbb{N}^n$ is 
$\mathcal{M}_{K_{n+1}^{a,b} - E}^{(n-2)}$. Also
${\rm PF}(K_{n+1}^{a,b} - E) = \emptyset$, as
$K_{n+1}^{a,b} - E$ is disconnected. Thus
\begin{eqnarray*}
 |{\rm sPF}(K_{n+1}^{a,b})| & = &|{\rm sPF}(K_{n+1}^{a,b} - E)| = 
\dim_{\mathbb{K}}\left( \frac{R}{\mathcal{M}_{K_{n+1}^{a,b} - E}^{(n-2)}} \right) \\
& = & {\rm Number ~of} ~ \lambda{\rm ~parking ~functions~ for}~\lambda =
 ((n-1)b,(n-2)b,\ldots,b,b)\\
& =& (n!) g_{n;n-2}(b) = b^n (n-1)^{n-1},
\end{eqnarray*}
where the polynomial $g_{n;n-2}(x)$ is given in the Remark \ref{Rem1}.
\hfill $\square$
\end{proof}
Let $K_{m+1,n}$ be the complete bipartite graph on  
$ V^{\prime} = \{0,1,\ldots,m\} \coprod \{m+1,\ldots,m+n\}$. Let 
$K_{m+1,n}^{a,b}$ be the complete bipartite multigraph on
$V^{\prime}$ (defined similar to $K_{n+1}^{a,b}$). More precisely, there
are $a$ number of edges in $K_{m+1,n}^{a,b}$
between the root $0$ and $j$, while
$b$ number of edges between $i$ and $j$, where
$i \in \{1,\ldots,m\}$ and $j \in \{m+1,\ldots,m+n\}$.
Let $E$ be the set of all edges of 
$K_{m+1,n}$ or $K_{m+1,n}^{a,b}$
through the root $0$.
\begin{proposition} We have
 $|{\rm sPF}(K_{m+1,n})|= |{\rm sPF}(K_{n+1,m})|$. More generally, 
\[ 
|{\rm sPF}(K_{m+1,n}^{a,b})|= |{\rm sPF}(K_{n+1,m}^{a,b})|. 
\]
\label{Prop4} 
\end{proposition}
\begin{proof}
 Let $E$ and $E^{\prime}$ be the set of all edges of
 $K_{m+1,n}^{a,b}$  and $K_{n+1,m}^{a,b}$
 through the root $0$, respectively.
On repeatedly applying the Lemma \ref{Lem2}, we see that 
\[ 
|{\rm sPF}(K_{m+1,n}^{a,b})| = |{\rm sPF}(K_{m+1,n}^{a,b} - E)|
\quad {\rm and} \quad  
|{\rm sPF}(K_{n+1,m}^{a,b})| = |{\rm sPF}(K_{n+1,m}^{a,b} - E^{\prime})|.
\]
Since graphs $K_{m+1,n}^{a,b} - E$ and 
$K_{n+1,m}^{a,b} - E^{\prime}$ are obtained from each other by 
interchanging vertices as $i \leftrightarrow n+i$  and
$m+j \leftrightarrow j$ (for $ i \in [m], j \in [n]$). Thus, 
$|{\rm sPF}(K_{m+1,n}^{a,b} - E)|
= |{\rm sPF}(K_{n+1,m}^{a,b} - E^{\prime})|$.

\hfill $\square$
\end{proof}

We now derive a combinatorial formula for 
$|{\rm sPF}(K_{m+1,n})| = |{\rm sPF}(K_{m+1,n} - E)$, using a free 
(non-minimal) cellular
resolution of the monomial ideal 
$\mathcal{M}_{K_{m+1,n} - E}^{(n-2)}$
supported on the order complex
${\mathbf{\Delta}} = \Delta(\Sigma_{m+n})$ of the Boolean poset
$\Sigma_{m+n}$ of non-empty subsets of $[m+n]$. 
Since $[m+n] = [m] \coprod [m+1,m+n]$,
every $A \in \Sigma_{m+n}$ has a disjoint
decomposition $A = A' \coprod A''$, where
$A' = A \cap [m]$ and $A'' = A \cap [m+1,m+n]$. The monomial label
$\mathbf{x}^{\alpha(A)}$ on the vertex $A$ is given by
\[
 \mathbf{x}^{\alpha(A)} = \begin{cases} \left( \prod_{j \in A'} x_j \right)^{n - |A''|}                
                               \left( \prod_{k \in A''} x_k \right)^{m -|A'|} & 
{\rm if}~ A' \ne [m]~{\rm and}~A'' \ne [m+1, m+n], \\
&\\
\left( \prod_{j \in A'} x_j \right) \left( \prod_{k \in A''} x_k \right) & {\rm otherwise}.
\end{cases}
\]
Define $\mu_{j,A}$ by setting $\mathbf{x}^{\alpha(A)} = \prod_{j \in A} x_j^{\mu_{j,A}}$. 
Clearly, the ideal $I_{{\mathbf{\Delta}}}$ generated by the monomial vertex labels 
$\mathbf{x}^{\alpha(A)}$; $A \in \Sigma_{m+n}$ is the monomial ideal 
$\mathcal{M}_{K_{m+1,n} - E}^{(n-2)}$. Since
$I_{\mathbf{\Delta}} = \mathcal{M}_{K_{m+1,n} - E}^{(n-2)}$
is an order monomial ideal in the sense of Postnikov and Shapiro \cite{PoSh}, 
the free complex $\mathbb{F}_{*}(\mathbf{\Delta})$ supported on the 
order complex ${\mathbf{\Delta}} = \Delta(\Sigma_{m+n})$ is a free resolution
of  $I_{\mathbf{\Delta}} = \mathcal{M}_{K_{m+1,n} - E}^{(n-2)}$.
Using the cellular resolution $\mathbb{F}_{*}(\mathbf{\Delta})$, 
the multigraded Hilbert series 
$H\left(\frac{R}{I_{\mathbf{\Delta}}}, \mathbf{x}\right)$
of $\frac{R}{I_{\mathbf{\Delta}}} = 
\frac{R}{\mathcal{M}_{K_{m+1,n} - E}^{(n-2)}}$ is given by
\begin{eqnarray}
\label{Eqn2} 
H\left(\frac{R}{I_{\mathbf{\Delta}}}, \mathbf{x}\right) & = & 
\frac{\sum_{i=0}^{m+n} (-1)^{i} \sum_{{\emptyset=A_0 \subsetneq A_1
\subsetneq \ldots \subsetneq A_i \subseteq [m+n]}} \prod_{\ell =1}^i 
\left( \prod_{j \in A_{\ell} \setminus A_{\ell-1}} 
x_j^{\mu_{j,A_{\ell}}} 
\right) }{(1-x_1) \cdots (1-x_n)}. 
\end{eqnarray}

For $ 1 \le i \le m+n$, let $\Gamma_i$ be the set of
order pairs $(\mathbf{s}, \mathbf{t})$ of $i$-tuples
$\mathbf{s}=(s_1,\ldots,s_i)$ and $\mathbf{t} = (t_1,\ldots,t_i)$ 
such that $0=s_0 \le s_1 \le \ldots \le s_i$, 
$0=t_0 \le t_1 \le \ldots \le t_i$, $s_i + t_i = m+n$ and
$s_{j-1}+t_{j-1} < s_j+t_j$ for $1 \le j \le i$. For 
$(\mathbf{s}, \mathbf{t}) \in \Gamma_i$, let 
$\mathfrak{T}(\mathbf{s}, \mathbf{t}) = 
\{ j \in [i] : s_j < m ~{\rm and}~ t_j < n \}$ and
set 
\[
 \mu(\mathbf{s}, \mathbf{t}) = 
 \prod_{j \in \mathfrak{T}(\mathbf{s},\mathbf{t})} ~ 
(n-t_j)^{s_j - s_{j-1}} (m-s_j)^{t_j - t_{j-1}}.
\]
 
\begin{proposition}
 The number of spherical $K_{m+1,n}$-parking functions is given by
\[
 |{\rm sPF}(K_{m+1,n})| = \sum_{i=1}^{m+n} (-1)^{m+n-i} 
\prod_{(\mathbf{s}, \mathbf{t}) \in \Gamma_i} { m \choose s_1,s_2-s_1, \ldots,s_i - s_{i-1}}
{ n \choose t_1,t_2-t_1, \ldots, t_i - t_{i-1}} \mu(\mathbf{s}, \mathbf{t}) .
\]
\label{Prop5} 
\end{proposition}

\begin{proof}
Since the Artinian quotient $\frac{R}{I_{\mathbf{\Delta}}} = 
\frac{R}{\mathcal{M}_{K_{m+1,n} - E}^{(n-2)}}$ has finitely 
many standard monomials, we have 
\[
 \dim_{\mathbb{K}} \left(\frac{R}{I_{\mathbf{\Delta}}} \right) = 
\dim_{\mathbb{K}} \left(\frac{R}{\mathcal{M}_{K_{m+1,n} - E}^{(n-2)}} \right) = 
H\left(\frac{R}{I_{\mathbf{\Delta}}}, \mathbf{1}\right),
\]
where $\mathbf{1}=(1,\ldots,1)\in \mathbb{N}^n$. Letting 
$\mathbf{x}=(x_1,\ldots,x_n) \rightarrow
(1,\ldots,1) = \mathbf{1}$ in the rational function expression (\ref{Eqn2}) 
of $H(\frac{R}{I_{\mathbf{\Delta}}}, \mathbf{x})$, and applying
L'Hospital's rule, we get 
\[
  \dim_{\mathbb{K}} \left(\frac{R}{I_{\mathbf{\Delta}}} \right) =  
\sum_{i=0}^{m+n} (-1)^{m+n-i} \sum_{{\emptyset=A_0 \subsetneq A_1
\subsetneq \ldots \subsetneq A_i = [m+n]}} \prod_{\ell =1}^i 
\left( \prod_{j \in A_{\ell} \setminus A_{\ell-1}} 
\mu_{j,A_{\ell}} \right),
\]
where the summation runs over all chains 
$\emptyset=A_0 \subsetneq A_1
\subsetneq \ldots \subsetneq A_i = [m+n]$ in $\Sigma_{m+n}$. 
Let $s_j = |A'_j|$ and $t_j = |A''_j|$. Then 
$(\mathbf{s},\mathbf{t}) \in \Gamma_i$. In this case,  the chain
$\emptyset=A_0 \subsetneq A_1
\subsetneq \ldots \subsetneq A_i = [m+n]$ is said to be 
of {\em type} $(\mathbf{s},\mathbf{t}) \in \Gamma_i$
and for such chains, we have $\prod_{\ell =1}^i 
\left( \prod_{j \in A_{\ell} \setminus A_{\ell-1}} 
\mu_{j,A_{\ell}} \right) = \mu(\mathbf{s},\mathbf{t})$. Further, the number 
of chains in $\Sigma_{m+n}$ of type $(\mathbf{s}, \mathbf{t})$ is precisely
${ m \choose s_1,s_2-s_1, \ldots,s_i - s_{i-1}}
{ n \choose t_1,t_2-t_1, \ldots, t_i - t_{i-1}}$.
As the graph $K_{m+1,n} - E$ is disconnected, 
${\rm PF}(K_{m+1,n} - E) = \emptyset $. Thus,
$\dim_{\mathbb{K}} \left(\frac{R}{\mathcal{M}_{K_{m+1,n} - E}^{(n-2)}} \right) = 
|{\rm sPF}({K_{m+1,n} - E})|$. In view of Proposition \ref{Prop3}, 
we get the desired formula. \hfill $\square$
\end{proof}

From Proposition \ref{Prop5}, we clearly have 
$|{\rm sPF}(K_{m+1,n})| = |{\rm sPF}(K_{n+1,m})|$.
Further, proceeding as in Proposition \ref{Prop5}, it can be shown that
$ |{\rm sPF}(K_{m+1,n}^{a,b})| = b^{m+n} |{\rm sPF}(K_{m+1,n})|$.

We now compute the number $|{\rm sPF}(G)|$ of spherical $G$-parking functions
for $G = K_{n+1} - \{e\}$, where $e$ is an edge not through the root $0$.
 We first consider the case $e = e_1$, where 
 $e_1=(1,n)$ is the edge joining $1$ and $n$.
As $\widetilde{{\rm sPF}}(G) \subseteq {\rm PF}(G)$ 
for $G = K_{n+1} - \{ e_1\}$,
on applying Corollary \ref{Cor3}, we get
an injective map $\phi_G : {\rm sPF}(G) \longrightarrow \mathcal{U}_n^{\prime}$,
where $\mathcal{U}_n^{\prime} = \mathcal{U}_{G - \{0\}}$ is the set of uprooted
trees on $[n]$ with  no edge between $1$ and $n$ (i.e., $1 \nsim n$).

\begin{theorem} For $n \ge 3$
 and $G = K_{n+1} - \{e_1\}$, the map 
$\phi_G : {\rm sPF}(G) \longrightarrow
 \mathcal{U}_n^{\prime}$ is a bijection.
\label{Thm5} 
\end{theorem}
\begin{proof}
Let $\mathcal{P} \in {\rm sPF}(G)$ and 
$\widetilde{\mathcal{P}} \in {\rm PF}(G)$ be the 
associated reduced spherical $G$-parking function. Choose
the largest vertex $r$ of $G - \{0\}$ such that 
$\widetilde{\mathcal{P}}(r) = 0$. We claim that 
$\widetilde{\mathcal{P}}(j) = 0 $ for some $j < r$. Otherwise, 
$\mathcal{P}(i) \ge 2, ~\forall~ i \in [n] \setminus \{r\}$, 
a contradiction to $\mathcal{P} \in {\rm sPF}(G)$. Now consider $r$
to be the root of the graph $G' = G - \{0\}$ on $[n]$. Then
$\widehat{\mathcal{P}}   
= \widetilde{\mathcal{P}} \mid_{[n] \setminus \{r\}}$ 
is a $G'$-parking function. On applying Theorem \ref{Thm1},
we get a spanning tree $\phi(\widehat{P})$ of the graph $G'$ with root $r$.
Since $\widehat{P}(i) \ge 1$ for $i >r$, 
all the edges $(r,i)$ are dampened. Hence, 
$\phi(\widehat{P})$ is a uprooted spanning tree of $G$ with root $r$. Define 
$\phi_G(\mathcal{P}) = \phi(\widehat{P})$. Clearly, 
$\phi_G: {\rm sPF}(G) \longrightarrow 
\mathcal{U}_n^{\prime}$ is injective.

We now show that $\phi_G$ is surjective.
 Let $T \in \mathcal{U}_n^{\prime}$ with root $r$. 
From Theorem \ref{Thm2}, the map 
$\phi_n : {\rm sPF}(K_{n+1}) \longrightarrow \mathcal{U}_n$
is bijective. Thus there exists 
$\mathcal{P} \in {\rm sPF}(K_{n+1})$ such that
$\phi_n(\mathcal{P}) = \phi(\widehat{\mathcal{P}}) = T$, where  
$\widetilde{\mathcal{P}}$
is the  reduced spherical parking 
function associated to $\mathcal{P}$ and $\widehat{\mathcal{P}}
= \widetilde{\mathcal{P}} \mid_{[n] \setminus \{r\}}$. 
\\
Claim : $ \mathcal{P} \in {\rm sPF}(G)$.

 Let $\mathcal{M}_{K_{n+1}} = \langle m_A : 
\emptyset \ne A \subseteq [n] \rangle$ and
$\mathcal{M}_{G} = \langle m_A^{\prime} : 
\emptyset \ne A \subseteq [n] \rangle$. Then
$m_A = m_A^{\prime}$ if either $\{1,n\} \subseteq A$ or
 $\{1,n\} \subseteq [n] \setminus A$.
Also, $x_1 m_A^{\prime} = m_A$  if
$\{1,n\} \cap A = \{1\}$ and  $x_n m_A^{\prime} = m_A$ 
if $\{1,n\} \cap A = \{n\}$. If
$ \mathcal{P} \notin {\rm sPF}(G)$, then there 
exists $A \subseteq [n]$ such that
$m_A^{\prime} \mid \mathbf{x}^{\mathcal{P}}$ but 
$m_B \nmid \mathbf{x}^{\mathcal{P}}$ for all 
$\emptyset \ne B \subsetneq [n]$.
We shall assume that $1 \in A$ but $n \notin A$. 
The other case, $n \in A$ but $1 \notin A$ 
is similar. 

Suppose $A = \{i_1,i_2,\ldots,i_t\}$ such that
 $1=i_1 < i_2 < \ldots < i_t < n$. As
$m_A = (\prod_{j \in A} x_j )^{n-t+1}$ and 
$m_A =x_1 m_A^{\prime} \nmid \mathbf{x}^{\mathcal{P}}$, 
we have $\mathcal{P}(1) = n-t$ and 
$\mathcal{P}(i_k) \ge n-t+1$ for $k=2,\ldots,t$.
Let $[n] \setminus A = \{r = j_1,j_2,\ldots,j_s\}$ 
such that $\mathcal{P}(j_1) \le 
\ldots \le \mathcal{P}(j_s)$. Then $s+t =n$. Since
 $\mathcal{P} \in {\rm sPF}(K_{n+1})$, 
we must have $\mathcal{P}(j_1) = 1, 
\mathcal{P}(j_2) < 2, \ldots, \mathcal{P}(j_s) < s$. 
This shows that $\widehat{\mathcal{P}}(j_2) = 0, \ldots, 
\widehat{\mathcal{P}}(j_s) < s-1$,
 $\widehat{\mathcal{P}}(1) =s-1$ and 
$\widehat{\mathcal{P}}(i_k) \ge s$, for $2 \le k \le t$.
Now we apply DFS algorithm to get spanning tree
 $\phi(\widehat{\mathcal{P}})$ with root $r$.
Starting from the root $r$, all the vertices
 $j_2, \ldots, j_s$ get burnt in the first $s-1$
steps. Also, whenever certain edges joining $j_l$ and $i_k$ 
 get dampened, it reduces the value
$\widehat{\mathcal{P}}(i_k)$ by $1$ for $k\ne 1$. Since
 $j_l= n$ for some $l$, after the $(s-1)^{th}$ step
 of DFS algorithm, $\widehat{\mathcal{P}}(1) = s-1$ 
and the reduced values of $\widehat{\mathcal{P}}(i_k)$ are all $\ge 1$
 for $2 \le k \le t$. The value 
$\widehat{\mathcal{P}}(1)$ reduces by at most $1$
 if the search backtracks from $j_l$ to $j_{l-1}$
 for $l \in \{2,\ldots,s\}$.
Again, as $j_l = n$ for some $l$ and $1\nsim n$, we see that
the reduced value of $\widehat{\mathcal{P}}(1)$ is
 $1$ even after the search backtracks
to the root $r=j_1$. Hence, $1$ is not a 
burnt vertex, a contradiction to
$\phi(\widehat{\mathcal{P}}) = T$. 
This proves the claim and the theorem. \hfill $\square$
\end{proof} 
\begin{remarks} {\rm
 \begin{enumerate}
  \item[{\rm (1)}] By renumbering vertices of $G$, we easily see that 
\[|{\rm sPF}(K_{n+1} - \{e\})| = 
|{\rm sPF}(K_{n+1} - \{e_1\})|=
 |\mathcal{U}_n^{\prime}|,
\]
for any edge $e$ between distinct vertices $i,j \in [n]$. 
\item[{\rm (2)}] Let $e=(n-1,n)$ be the
 edge in $K_{n+1}$ joining 
$n-1$ and $n$ and $G = K_{n+1} - \{e\}$. 
For $n \ge 3$, the injective map 
$\phi_G :{\rm sPF}(G) \longrightarrow 
\mathcal{U}_{G - \{0\}}$
need not be surjective.  In fact, 
for $n=4$, 
 $|{\rm sPF}(K_{4+1} - \{e\})| = 12$ but the 
number of  uprooted trees on $[4]$ with no edge between $3$ and $4$ is 
exactly $17$. 
\item[{\rm (3)}] Let $E_1=E(1,n)=\{e_{1n}^0,e_{1n}^1,\ldots,e_{1n}^{b-1}\}$ 
be the set of all edges joining $1$ and $n$ in the complete multigraph
$K_{n+1}^{a,b}$. We recall that $\mathcal{U}_n^b$ is the set of uprooted trees
$T$ on 
$[n]$ with a label $\ell : E(T) \rightarrow \{0,1,\ldots,b-1\}$ on its edges
and a weight $\omega(r) \in \{0,1,\ldots,b-1\}$ on its root $r$. 
Let $ \mathcal{U}_n^{\prime b } = \{ T \in \mathcal{U}_n^b : 1 \nsim n~
{\rm in~ T}\}$.
Then using Theorem \ref{Thm4}, 
the bijection of the Theorem \ref{Thm5} can be extended to a bijection
\[
\phi_n^b : {\rm sPF}(K_{n+1}^{a,b} - E_1) \longrightarrow \mathcal{U}_n^{\prime b}.
\]
 In particular, 
$|{\rm sPF}(K_{n+1}^{a,b} - E_1)| = |\mathcal{U}_n^{\prime b}| = b^n |\mathcal{U}_n^{\prime}|$.
 \end{enumerate}
}
\label{Rem2} 
\end{remarks}

We now determine the number
$|\mathcal{U}_n^{\prime}|$  of uprooted trees on
$[n]$ with $1 \nsim n$.

Let $\mathcal{T}_{n,0}$ be the set of labelled trees 
on $[n]$ such that the root has no child (or son) with smaller
labels. Let $\mathcal{A}_n$ be the set of labelled rooted-trees
on $[n]$ with  a non-rooted leaf $n$. Chauve, Dulucq and Guibert \cite{CDG}
constructed a bijection $\eta : \mathcal{T}_{n,0} \rightarrow \mathcal{A}_n$.
As earlier, let $\mathcal{U}_n$ be the set of uprooted trees on
$[n]$. Also, let $\mathcal{B}_n$ be the set of labelled rooted-trees on 
$[n]$ with a non-rooted leaf $1$. We see that there are bijections
$\mathcal{U}_n \rightarrow \mathcal{T}_{n,0}$ and 
$\mathcal{B}_n \rightarrow  \mathcal{A}_n$ obtained by simply changing
label $i$ to $n-i+1$ for all $i$. The bijection 
$\eta : \mathcal{T}_{n,0} \rightarrow \mathcal{A}_n$ induces a bijection
$\psi : \mathcal{U}_n \rightarrow \mathcal{B}_n$. For sake of 
completeness, we briefly describe construction of the
bijection $\psi$ essentially as in \cite{CDG}. 

Let $T \in \mathcal{U}_n$ with root $r$. Note that $r \ne 1$. \\
Step (1) : Consider a maximal increasing subtree
$T_0$ of $T$ containing $1$. Let $T_1, \ldots, T_l$ be the subtrees 
(with at least one edge) of $T$ obtained by deleting edges in $T_0$.
Let $r_i$ be the root of $T_i$ for $ 1 \le i \le l$. 
The root $r$ of $T$ must be a root of one of the subtrees $T_i$.
Let $r_j=r$. Then $1$ is a leaf of $T_j$. \\
Step (2) : If $T_0$ has $m$ vertices, then $T_0$ is 
determined by an increasing tree $\overline{T_0}$ on
$[m]$ and a set $S_0$ of labels on $T_0$. We write
 $T_0 = (\overline{T_0},S_0)$.\\
Step (3) : Let
$\overline{S_0} = (S_0 \setminus \{1\}) \cup \{r\}$. Then
$(\overline{T_0},\overline{S_0})$ determines an increasing
subtree $\widetilde{T_0}$ with root $r'= \min\{\overline{S_0}\}$.
Graft $T_j$ on the increasing subtree $\widetilde{T_0}$ at the root $r$
and obtain a tree $T_j^{\prime}$. Now graft $T_i$ ($i \ne j$) on 
$T_j^{\prime}$ at $r_i$ and obtain a tree $T^{\prime}$ with root 
$r'$. Also note that $1$ is a non-rooted leaf of $T^{\prime}$.

All the above steps can be reversed, thus $\psi(T) = T^{\prime}$
defines a bijection $\psi : \mathcal{U}_n \rightarrow \mathcal{B}_n$.

\begin{lemma}
$ |\mathcal{U}_n| = |\mathcal{B}_n| = (n-1)^{n-1}$.
\label{Lem3}  
\end{lemma}
\begin{proof} The bijection $\psi : \mathcal{U}_n \rightarrow \mathcal{B}_n$ 
gives  $|\mathcal{U}_n| = |\mathcal{B}_n|$.
 The number of labelled rooted-trees on $\{2,3,\ldots,n\}$ by Cayley's formula is
$(n-1)^{n-2}$. Any tree in $\mathcal{B}_n$ is obtained uniquely
by attaching $1$ to any node
$i$ of a labelled rooted tree on $\{2,3,\ldots,n\}$. Since there are 
exactly $n-1$ possibility for $i$, we have
$|\mathcal{B}_n|=(n-1)^{n-2} (n-1) =(n-1)^{n-1}$. \hfill $\square$ 
\end{proof}

For $n \ge 3$, let $\mathcal{U}_n^{\prime} = \{ T \in \mathcal{U} : 1 \nsim n ~{\rm in~} T \}$. 
We shall determine the image $\psi(\mathcal{U}_n^{\prime}) \subseteq \mathcal{B}_n$
of $\mathcal{U}_n^{\prime}$ under the bijection $\psi : \mathcal{U}_n \rightarrow \mathcal{B}_n$. Let 
$\mathcal{B}_n^{\prime} = \{T' \in \mathcal{B}_n:
1 \nsim n ~{\rm in}~T'\}$. Set 
\begin{eqnarray*}
\mathcal{A} & = & \{T' \in \mathcal{B}_n^{\prime}:
~{\rm root}(T')=r'=n\}, \\
\mathcal{B}^{\prime} & = & \{ T' \in \mathcal{B}_n^{\prime} : ~{\rm root}(T')=r'\neq n ~{\rm with}~
r' \sim n ~{\rm and}~1~{\rm is~a~descendent~of~} n\}, \\
\mathcal{B}^{\prime \prime} & = & \{T' \in 
\mathcal{B}_n^{\prime} : {\rm root}(T')=r' \ne n ~
{\rm with}~r' \nsim n \}.
\end{eqnarray*}

\begin{lemma}
$\psi(\mathcal{U}_n^{\prime}) = \mathcal{A} \coprod 
\mathcal{B}^{\prime} \coprod 
\mathcal{B}^{\prime \prime}$.
 \label{Lem4}
\end{lemma}
\begin{proof}
 Let $T' \in \mathcal{B}_n$. Then there is a unique $T \in \mathcal{U}_n$ such that
$T' = \psi(T)$. Let $r$ and $r'$ be the roots of $T$ and $T'$, respectively. Clearly,
$r \ne 1$. Let ${\rm Son}_T(1)$ be the set of sons of $1$ in $T$. Then from the construction
 of $T'=\psi(T)$, $r' = \min\{ \{r\} \cup {\rm Son}_T(1)\}$. Also, the leaf $1$ in $T'$ is
adjacent to $j$ if and only if $j = {\rm par}_T(1)$ is the parent of $1$ in $T$. This 
shows that $1 \nsim n$ in $T$ if and only if $1 \nsim n$ in $T'$. Hence, $\psi(\mathcal{U}_n^{\prime})
\subseteq \mathcal{B}_n^{\prime}$. 
Further, we see that 
$r' = n$ if and only if $1$ is already a leaf in $T$, and in this 
case, $T' = \psi(T) = T$. 
In other words, $\mathcal{A} \subseteq 
\mathcal{U}_n^{\prime}$ and $\psi(T) = T$ for all
$T \in \mathcal{A}$. 

If $T' \in \mathcal{B}^{\prime \prime}$, then 
the unique $T \in \mathcal{U}_n$ with $\psi(T) = T'$
must have $1 \nsim n$ in $T$, that is, $T \in 
\mathcal{U}_n^{\prime}$. Now we consider the remaining case. Let $T' \in \mathcal{B}_n^{\prime}$ with
${\rm root}(T')=r' \ne n $ and $ r' \sim n$ in 
$T'$. We shall show that $\psi(T) = T'$ for $T \in 
\mathcal{U}_n^{\prime}$ if and only if 
$1$ is a descendent of $n$ in $T'$
(or equivalently, $T' \in \mathcal{B}^{\prime}$). Consider the
maximal increasing subtree $T_0^{\prime}$ of $T'$
containing the root $r'$. If $1$ is a descendent 
of a leaf $r_j'$ of $T_0^{\prime}$, then the maximal
increasing subtree $T_0$ of $T$ containing $1$ 
is obtained by replacing $r_j'$ with $1$ in the vertex set of $T_0^{\prime}$ and labeling it as indicated
 in Step (2) of the construction of $\psi$.
 Clearly, $r_j' = r$ is the root of $T$. If $r_j'=r
 \ne n$, then $1 \sim n$ in $T$ as $r' \sim n$ in $T'$.
 Thus, if $r_j' \ne n$, i.e., $1$ is not a descendent
 of $n$ in $T'$, then $T' \notin 
 \psi(\mathcal{U}_n^{\prime})$. On the other hand,
 if $r_j' = n$, i.e., $1$ is a 
 descendent of $n$ 
 in $T'$ with $1 \nsim n$, then ${\rm root}(T) = r=n$ and $1 \nsim n$ in
 $T$. 
\hfill $\square$
\end{proof}
\begin{proposition} For $n \ge 3$, we have
 $|\mathcal{U}_n^{\prime}| = (n-1)^{n-3}(n-2)^2 $.
\label{Prop6} 
\end{proposition}
\begin{proof} By Lemma \ref{Lem4}, we have 
$|\mathcal{U}_n^{\prime}|= |\psi(\mathcal{U}_n^{\prime})| =
|\mathcal{A}| + |\mathcal{B}^{\prime}| 
+ |\mathcal{B}^{\prime \prime}|$.
 First we enumerate the subset
 $\mathcal{A} = \{ T' \in \mathcal{B}_n^{\prime} : 
~{\rm root}(T')=r'=n\}$.
The number of labelled trees on $\{2,3,\ldots,n\}$ with root $n$ is $(n-1)^{n-3}$. 
Since any tree in 
$\mathcal{A}$ is uniquely obtained by attaching $1$ to any node $i \in \{2, \ldots, n-1\}$
of a labelled tree on $\{2,\ldots,n\}$ with root $n$, we have
$|\mathcal{A}| = (n-1)^{n-3}(n-2)$. 

Let us consider the subset $\mathcal{C} =\{ T' \in \mathcal{B}_n^{\prime} : {\rm root}(T')= r'\ne n \}
\subseteq \mathcal{B}_n^{\prime}$.
Clearly, $\mathcal{B}= \mathcal{B}^{\prime} \coprod
\mathcal{B}^{\prime \prime} \subseteq \mathcal{C}$.
 The enumeration of $\mathcal{C}$ is similar
to that of $\mathcal{A}$, except
now the root $r' \in \{2,\ldots,n-1\}$ can take any one of the 
$n-2$ values. Thus $|\mathcal{C}| = (n-1)^{n-3}(n-2)^2$. We
can easily construct a 
bijective correspondence between $\mathcal{A}$ and $\mathcal{C} \setminus \mathcal{B}$.
Let $T'\in \mathcal{A}$. Then $1 \nsim n$ in $T'$ and ${\rm root}(T')=n$. 
Consider the unique path from the root $n$ to the 
leaf $1$ in $T'$. As $ 1 \nsim n$ in $T'$, the 
child $\tilde{r}$ of $n$ lying on this unique path is
different from $1$. Let $\tilde{T'}$ be rooted tree consisting of the tree $T'$ with the new 
root $\tilde{r}$. As ${\rm root}(\tilde{T'}) = 
\tilde{r} \ne n, ~\tilde{r} \sim n$ and 
$1$ is not a descendent of $n$ in $\tilde{T'}$,
we have  $\tilde{T'} \in \mathcal{C} \setminus 
\mathcal{B}$.
The mapping $T' \mapsto \tilde{T'}$
 from $\mathcal{A}$ to $\mathcal{C} \setminus \mathcal{B}$ is clearly a bijection.
 If $\tilde{T'} \in 
 \mathcal{C} \setminus \mathcal{B}$, then 
 ${\rm root}(\tilde{T'}) = \tilde{r} \ne n, ~\tilde{r}
 \sim n$ and $1$ is not a descendent of $n$ in 
 $\tilde{T'}$. Now unique $T' \in 
 \mathcal{A}$ that maps to $\tilde{T'}$
 is the rooted tree obtained from $\tilde{T'}$ 
by taking $n$ as the new
 root. Thus  
$|\mathcal{A}| = |\mathcal{C} \setminus \mathcal{B}|$ and hence,
$|\mathcal{U}_n^{\prime}| = |\mathcal{C}| = (n-1)^{n-3}(n-2)^2$.
\hfill $\square$
\end{proof}

\begin{theorem}
 Let $e$ be an edge of $K_{n+1}$ joining distinct
vertices $i, j \in [n]$. For $n \ge 3$, 
the number of spherical parking functions of $K_{n+1} \setminus \{e\}$  is given by
\[|{\rm sPF}(K_{n+1} \setminus \{e\})| = 
 |\mathcal{U}_n^{\prime}| = (n-1)^{n-3}(n-2)^2.
\]
\label{Thm6} 
\end{theorem}
\begin{proof}
 In view of Theorem \ref{Thm5} and Remarks \ref{Rem2}, the result follows. \hfill $\square$ 
\end{proof}

 {\bf Acknowledgements} : The second author is thankful to  
 MHRD, Government of India for financial support.

\end{document}